\documentclass{scrartcl}
\usepackage[T1]{fontenc}
\usepackage{lmodern}
\usepackage{amsmath}
\usepackage{amsthm}
\usepackage{amsfonts}
\usepackage{amssymb}
\usepackage[shortlabels]{enumitem}
\usepackage{mathrsfs}
\usepackage{hyperref}
\usepackage{mathtools}
\usepackage{cite}
\usepackage{orcidlink}

\hypersetup{
breaklinks=true,
colorlinks=true,
linkcolor=blue,
citecolor=blue,
urlcolor=blue,
filecolor=blue,
}

\numberwithin{equation}{section} 

\newcommand{\R}{\mathbb{R}}

\newcommand{\N}{\mathbb{N}}
\newcommand{\calB}{\mathcal{B}}

\newcommand{\calD}{\mathcal{D}}
\newcommand{\calE}{\mathcal{E}}
\newcommand{\calF}{\mathcal{F}}
\newcommand{\calG}{\mathcal{G}}
\newcommand{\calH}{\mathcal{H}}

\newcommand{\calJ}{\mathcal{J}}
\newcommand{\calK}{\mathcal{K}}
\newcommand{\calL}{\mathcal{L}}

\newcommand{\calN}{\mathcal{N}}

\newcommand{\calP}{\mathcal{P}}

\newcommand{\calR}{\mathcal{R}}

\newcommand{\calX}{\mathcal{X}}
\newcommand{\calY}{\mathcal{Y}}

\newcommand{\bbA}{\mathbb{A}}
\newcommand{\bbB}{\mathbb{B}}

\newcommand{\bbD}{\mathbb{D}}
\newcommand{\bbE}{\mathbb{E}}

\newcommand{\bbG}{\mathbb{G}}
\newcommand{\bbH}{\mathbb{H}}
\newcommand{\bbI}{\mathbb{I}}

\newcommand{\bbK}{\mathbb{K}}
\newcommand{\bbL}{\mathbb{L}}

\newcommand{\bbP}{\mathbb{P}}
\newcommand{\bbQ}{\mathbb{Q}}

\newcommand{\bbS}{\mathbb{S}}
\newcommand{\bbT}{\mathbb{T}}

\newcommand{\bbW}{\mathbb{W}}

\newcommand{\bE}{\boldsymbol{E}}

\newcommand{\ba}{\boldsymbol{a}}
\newcommand{\bb}{\boldsymbol{b}}

\newcommand{\bf}{\boldsymbol{f}}

\newcommand{\bn}{\boldsymbol{n}}

\newcommand{\bv}{\boldsymbol{v}}

\newcommand{\rmD}{\mathrm{D}}

\newcommand{\rmN}{\mathrm{N}}

\newcommand{\wt}[1]{\widetilde{#1}}
\newcommand{\wh}[1]{\widehat{#1}}
\newcommand{\ol}[1]{\overline{#1}}

\newcommand{\jder}[1]{\mathring{#1}}
\newcommand{\dd}{\,\mathrm{d}}

\DeclareMathOperator{\dv}{div}

\DeclareMathOperator*{\argmin}{argmin}
\DeclareMathOperator{\dom}{dom}

\DeclareMathOperator{\sph}{sph}
\DeclareMathOperator{\dev}{dev}
\DeclareMathOperator{\sym}{sym}
\DeclareMathOperator{\tr}{tr}
\newcommand{\tow}{\mathrel{\rightharpoonup}}
\newcommand{\tows}{\mathrel{\xrightharpoonup{*}}}
\newcommand{\np}[1]{(#1)}
\newcommand{\nb}[1]{[#1]}
\newcommand{\lp}[1]{\big(#1\big)}
\newcommand{\lb}[1]{\big[#1\big]}
\newcommand{\Lp}[1]{\Big(#1\Big)}
\newcommand{\Lb}[1]{\Big[#1\Big]}
\newcommand{\llp}[1]{\bigg(#1\bigg)}

\newcommand{\set}[1]{{\{#1\}}}

\newcommand{\setL}[1]{{\biggl\{#1\biggr\}}}
\newcommand{\setc}[2]{{\{#1 \,\vert\, #2\}}}
\newcommand{\setcl}[2]{{\bigl\{#1 \,\big\vert\, #2\bigr\}}}

\newcommand{\norm}[1]{\lVert#1\rVert}

\newcommand{\abs}[1]{{\lvert #1 \rvert}}
\newcommand{\absl}[1]{{\bigl\lvert #1 \big\rvert}}

\newcommand{\BV}{\mathrm{BV}}
\newcommand{\TV}{\mathrm{TV}}
\makeatletter
\newcommand{\tcolon}{%
  \mathrel{\mathpalette\tcolon@\relax}%
}
\newcommand{\tcolon@}[2]{%
  \sbox\z@{$\m@th#1:$}%
  \vbox to\ht\z@{%
    \hbox{$\m@th#1.$}%
    \vss
    \hbox{$\m@th#1.$}%
    \vss
    \hbox{$\m@th#1.$}%
  }%
}
\makeatother

\theoremstyle{plain}
\newtheorem{theorem}{Theorem}[section]
\newtheorem{lemma}[theorem]{Lemma}
\newtheorem{proposition}[theorem]{Proposition}

\theoremstyle{definition}
\newtheorem{definition}[theorem]{Definition}
\theoremstyle{remark}
\newtheorem{remark}[theorem]{Remark}

\begin{document}
\title{Solution concepts for a model of visco-elasto-plasticity with slight compressibility}

\author{Thomas Eiter%
\footnote{Freie Universit\"at Berlin, Department of Mathematics and Computer Science, Arnimallee 14, 14195 Berlin, Germany.}
\textsuperscript{\,,}%
\footnote{%
Weierstrass Institute for Applied Analysis and Stochastics,
Anton-Wilhelm-Amo-Str.\,39, 10117 Berlin, Germany.
\\
Email: {\texttt{thomas.eiter@wias-berlin.de}}}%
\ \,\orcidlink{0000-0002-7807-1349}
}

\date{}
\maketitle

\begin{abstract}
We study a model
for the deformation of a visco-elasto-plastic material 
that is nearly incompressible.
It originates from geophysics, is given in the Eulerian description and 
combines a Kelvin--Voigt rheology in the spherical part
with a Jeffreys-type rheology in the deviatoric part. 
Despite a constant density, 
the model allows for non-isochoric deformation 
and the propagation of pressure waves.
An additive decomposition of the strain rate 
into elastic and inelastic parts 
leads to an evolution equation for the small elastic strain,
which is coupled with an adapted momentum equation.
As plasticity is modeled through a non-smooth dissipation potential,
we introduce a weak formulation in terms of a variational inequality.
Since the well-posedness in such a weak setting is out of reach, 
we study two possible modifications: 
the regularization in terms of stress diffusion,
and the relaxation of the solvability concept 
by transition to energy-variational solutions.
In both cases, solutions are constructed by the same time-discrete scheme,
consisting of solving a saddle-point problem in each time step.
\end{abstract}


\section{Introduction}

The motion of tectonic plates
due to convective processes within the Earth's mantle 
is a well-accepted theory in geology. 
One way to describe this deformation of the Earth's lithosphere,
which happens on large time scales,
is to model rock as
a visco-elasto-plastic material~\cite{MoresiDufourMuehlhaus2002,GeryaYuen2007,HerrenGeryaVand2017IRSDF, Preussetal2019}.
Since there seems to be no reasonable choice for the reference configuration 
in this geophysical framework, 
a description in terms of the current configuration is natural,
that is, an Eulerian description.
The density varies merely on small scales,
which motivates to treat
it as a constant during the evolution.
However, this simplification would lead to an incompressibility constraint,
which does not allow the model to feature the propagation of longitudinal waves.
To combine both characteristics, we follow an approach first employed by Temam~\cite{Temam1968,Temam1969},
and consider an additional internal force in the momentum equation.
In summary, the studied model combines this setting of slight compressibility
with a visco-elasto-plastic rheology.

Let $T>0$ be a finite time horizon, and let 
$\Omega\subset\R^3$ be a bounded domain.
We consider the system
\begin{subequations}
\label{eq:system}
\begin{align}
\rho(\partial_t\bv+\bv\cdot\nabla \bv)
&=\dv (\bbS(\bv)+\bbT)-\frac{\rho}{2}(\dv \bv)\bv+\bf
&&\text{in }(0,T)\times\Omega,
\label{eq:sys.mom}
\\
\bbT 
&=D\varphi(\bbE)+\varphi(\bbE)\bbI
&&\text{in }(0,T)\times\Omega,
\label{eq:sys.stress}
\\
\jder{\bbE}+\partial\calP(\dev\bbT)&\ni\bbD(\bv)
&&\text{in }(0,T)\times\Omega.
\label{eq:sys.strain}
\end{align}
The free variables are 
the Eulerian velocity $\bv\colon(0,T)\times\Omega\to\R^3$
and the small elastic strain $\bbE\colon(0,T)\times\Omega\to\R^{3\times3}_{\sym}$.
The mass density $\rho>0$ is assumed to be constant, 
and $\bf\colon(0,T)\times\Omega\to\R^3$ 
denotes an external body force.
Equation~\eqref{eq:sys.mom} is the momentum equation,
enriched with an additional internal force $\frac{\rho}{2}(\dv\bv)\bv$.
The stress tensor
is the sum of a linear viscous stress 
$\bbS(\bv)$
and a second stress tensor $\bbT$
determined from $\bbE$ through the stress-strain relation~\eqref{eq:sys.stress}.
Here, $\varphi\colon\R^{3\times 3}_{\sym}\to\R$ is the potential of the stored elastic energy.
Equation~\eqref{eq:sys.strain} 
thus describes a Maxwell-type rheology,
which corresponds 
to an additive decomposition of the strain rate 
$\bbD(\bv)=\frac{1}{2}(\nabla\bv+\nabla\bv^\top)$,
into an elastic and an inelastic part. 
The elastic strain rate $\jder{\bbE}$ 
is defined
by the Zaremba--Jaumann rate
\[
\jder{\bbE}
=\partial_t\bbE + \bv\cdot\nabla\bbE + \bbE \bbW(\bv)-\bbW(\bv)\bbE,
\qquad
\bbW(\bv)
=\frac{1}{2}(\nabla\bv-\nabla\bv^\top),
\]
while the inelastic strain rate is given by the subdifferential $\partial \calP(\dev\bbT)$
of a convex (dual) dissipation potential $\calP$,
which models viscous and plastic effects.
Without loss of generality, we assume $\calP(0)=0$.
As $\calP$ only acts on the
deviatoric part of $\bbT$,
it does not affect the volumetric strain.
Therefore, the model combines a Kelvin--Voigt rheology in
the volumetric part 
with a Jeffreys-type rheology in the deviatoric part.
We refer to Section~\ref{sec:modeling}
for a detailed derivation of the model.

We complement the system~\eqref{eq:sys.mom}--\eqref{eq:sys.strain}
with the boundary and initial conditions
\begin{align}
\bv&=0 
&&\text{on }(0,T)\times\Gamma_\mathrm{D},
\label{eq:dirichletbdry}
\\
\bv\cdot\bn &=0, \ [(\bbT+\bbS(\bv))\bn]_\tau=0
&&\text{on }(0,T)\times\Gamma_\mathrm{N},
\label{eq:neumannbdry}
\\
(\bv,\bbE)(0)&=(\bv_0,\bbE_0) 
&&\text{in }\Omega
\label{eq:initial}
\end{align}
\end{subequations}
for given initial data $(\bv_0,\bbE_0)$.
The boundary $\partial\Omega$ is assumed to have Lipschitz regularity,
and $\Gamma_\mathrm{D}\subset\partial\Omega$ and $\Gamma_\mathrm{N}=\partial\Omega\setminus\Gamma_\mathrm{D}$
denote disjoint parts of the boundary,
where we prescribe no-slip and perfect-slip boundary conditions, respectively.
Here $\bn\colon\partial\Omega\to\R^3$ denotes the outer unit normal, and $[\ba]_\tau$ denotes the 
tangential part of a vector field $\ba\colon\partial\Omega\to\R^3$,
that is,
$[\ba]_\tau
=\ba-(\ba\cdot\bn)\bn$. 

One delicate feature of the model~\eqref{eq:system}
is the convex potential $\calP$,
which is a nonlinear and non-smooth functional in general.
In~\cite{MoresiDufourMuehlhaus2002,HerrenGeryaVand2017IRSDF, GeryaYuen2007,Preussetal2019},
the visco-plastic behavior of rock is modeled by the subdifferential $\partial\calP$ of
\begin{equation}
\label{eq:disspot.example}
\mathcal P(\bbA)
=\int_\Omega P(\bbA(x))\dd x,
\qquad 
P(\bbA)
=\begin{cases}
\frac{\nu}{2}|\bbA|^2 
&\text{if }|\bbA|\leq \sigma_\mathrm{yield},
\\
+\infty
&\text{if }|\bbA|> \sigma_\mathrm{yield},
\end{cases}
\end{equation}
for some shear viscosity $\nu>0$
and a yield stress $\sigma_{\mathrm{yield}}>0$,
which combines linear viscoelasicity for `small' stresses
with plastic deformation for `large' stresses.
As the subdifferential $\partial\calP(\dev\bbT)$ is multi-valued in general,
equation~\eqref{eq:sys.strain} is indeed a differential inclusion
and not a mere equality.
Using the definition of the subdifferential,
one can introduce a weak formulation of~\eqref{eq:system}
in terms of a variational inequality, 
which intertwines the standard weak formulation 
with the energy inequality,
see Subsect.\,\ref{subsec:weakform}.

Under the assumption of incompressibility
and for a quadratic stored-energy functional $\varphi$,
a similar weak formulation was used 
in~\cite{EiterHopfMielke2022,EiterHopfLasarzik2023},
where the existence and the weak-strong uniqueness of such solutions was shown
after a regularization.
The existence result was recently extended in~\cite{ChengLasarzikThomas2025CahnHillardVisco}
to the case of a two-phase flow.
However, in all these works, weak solutions are only obtained
after introduction of a diffusion term in~\eqref{eq:sys.strain},
which increases the regularity of solutions and gives
additional compactness properties
useful for the existence theory.

An analogous approach was applied in~\cite{roubicek_timediscviscoel_2025},
where a compressible visco-elastic rheology was studied.
More precisely, \eqref{eq:sys.stress} and~\eqref{eq:sys.strain}
were considered for the (smooth) quadratic dissipation potential
$\calP(\dev\bbT)=\frac{\nu}{2}|\dev\bbT|^2$ with some shear viscosity $\nu>0$,
and coupled with the classical continuity and momentum equations.
The existence of weak solutions was shown after introducing a viscous hyperstress
in the momentum equation.
This improves the regularity of the velocity field, 
whereby higher-order spatial regularity of the density and the elastic strain
is inherited from the initial values.

Therefore, the existence of weak solutions to~\eqref{eq:system}
also seems to require a modification
by the introduction of additional terms.
Here we follow the idea
from~\cite{EiterHopfMielke2022,EiterHopfLasarzik2023, ChengLasarzikThomas2025CahnHillardVisco}
and introduce a diffusive term
in the transport-type equation~\eqref{eq:sys.strain}.
For a diffusion coefficient $\gamma>0$,
we consider
\begin{subequations}
\label{eq:system.reg}
\begin{align}
\rho(\partial_t\bv+\bv\cdot\nabla \bv)
&=\dv (\bbS(\bv)+\bbT)-\frac{\rho}{2}(\dv \bv)\bv+\bf
&&\text{in }(0,T)\times\Omega,
\label{eq:sys.mom.reg}
\\
\bbT 
&= D\varphi(\bbE)+\varphi(\bbE)\bbI
&&\text{in }(0,T)\times\Omega,
\label{eq:sys.stress.reg}
\\
\jder{\bbE}+\partial\calP(\dev \bbT)&=\bbD(\bv) + \gamma\Delta [D\varphi(\bbE)]
&&\text{in }(0,T)\times\Omega,
\label{eq:sys.strain.reg}
\\
\bv&=0 
&&\text{on }(0,T)\times\Gamma_\mathrm{D},
\label{eq:dirichletbdry.reg}
\\
\bv\cdot\bn &=0, \ [(\bbT+\bbS(\bv))\bn]_\tau=0
&&\text{on }(0,T)\times\Gamma_\mathrm{N},
\label{eq:neumannbdry.reg}
\\
\bn\cdot\nabla [D\varphi(\bbE)]&=0 
&&\text{on }(0,T)\times\partial\Omega,
\label{eq:sys.bdry.E.reg}
\\
(\bv,\bbE)(0)&=(\bv_0,\bbE_0) 
&&\text{in }\Omega.
\label{eq:initial.reg}
\end{align}
\end{subequations}
Then~\eqref{eq:sys.strain.reg} is of parabolic type,
and~\eqref{eq:sys.bdry.E.reg} are the corresponding natural boundary conditions.
This regularization enables us to show the existence of weak solutions,
see Theorem~\ref{thm:existence.gamma} below.
However, to ensure convexity of the associated contribution to the dissipation,
we restrict the analysis to the case of a quadratic stored elastic energy $\varphi$ here. 
Then the diffusion term in~\eqref{eq:sys.strain.reg}
reduces to a linear term in $\bbE$,
and $\bbE$ appears in the energy-dissipation balance
in a quadratic way, compare~\eqref{eq:enin.reg} below.
To regularize a viscoelastic flow in terms of stress diffusion 
is a nowadays common approach~\cite{BMPS2018PDEA,BaBuMa21LDET, EiterHopfMielke2022,ChengLasarzikThomas2025CahnHillardVisco},
and it can be derived from thermodynamic considerations~\cite{MPSS18TVRT}.

One regards the regularized system~\eqref{eq:system.reg} 
for small $\gamma>0$
as an approximation of the original system~\eqref{eq:system}.
Therefore, it is natural to ask
what happens to solutions to~\eqref{eq:system.reg} as $\gamma\to0$.
Suitable a priori bounds indicate the convergence towards a limit object,
at least along a subsequence,
and the question is in which way 
this limit is connected to the original problem~\eqref{eq:system}.
Usually, the limit is not a weak solution,
but one arrives at a solution in a further generalized sense.

The study of such generalized solvability concepts,
which further relax the notion of weak solutions,
is nowadays quite common, in particular in fluid dynamics.
For instance,
DiPerna and Majda~\cite{DiPernaMajda1987OscConcIncomprFluid} introduced
the concept of \textit{measure-valued solutions} to the incompressible Euler equations.
A different approach goes back to P.-L.\,Lions~\cite[Sec.\,4.4]{Lions2013FM1},
who defined \textit{dissipative solutions} to the Euler equations
in terms of a relative energy inequality.
In both cases,
those solutions were constructed as the inviscid limit of weak solutions
to the Navier--Stokes equations.
In~\cite{EiterHopfLasarzik2023} and~\cite{ChengLasarzikThomas2025CahnHillardVisco},
the approach by Lions was adapted to the aforementioned systems of
incompressible visco-elasto-plasticity with one and two phases, respectively.
There, the solvability concept was based on a relative energy inequality,
and solutions were constructed as limits of weak solutions to
the regularized system.
While such solutions are also called dissipative solutions in~\cite{ChengLasarzikThomas2025CahnHillardVisco},
the name \textit{energy-variational solutions} was used in~\cite{EiterHopfLasarzik2023}.
However, the latter notion of solution 
has been further refined in various works 
since then~\cite{EitLas24envarhyp,ALR24envarvisc,LasRei23envarEL,Las24envarMDinc, Las25envarstr,EiterSchindler25,EiterLasarzikSliwinski_exselenvar}.
We follow this approach and introduce a corresponding notion of energy-variational solutions to problem~\eqref{eq:system},
see Subsect.\,\ref{subsec:envarsol},
where we adapt this refined solvability concept to our framework.

To show the existence of energy-variational solutions to~\eqref{eq:system}, 
we could proceed similarly to~\cite{EiterHopfLasarzik2023,ChengLasarzikThomas2025CahnHillardVisco}
and pass to the limit $\gamma\to0$ with weak solutions to~\eqref{eq:system.reg}.
Instead, we follow the approach 
from~\cite{EitLas24envarhyp,ALR24envarvisc}
and obtain solutions $(\bv,\bbE)$ to~\eqref{eq:system}
as the limit $N\to\infty$ 
of time-discrete approximate solutions $(\bv_N,\bbE_N)$,
which are constructed by iteratively solving 
a saddle-point problem at each time step.
To this end, one introduces a certain \textit{regularity weight}
that gives sufficient convexity properties to render the saddle-point problem solvable
and to pass to the limit by weak convergence.
As it was observed in~\cite{EiterLasarzikSliwinski_exselenvar}, 
the two steps are, to some extent, independent of each other,
which allows to use two different regularity weights.

In this article, we apply this iterative scheme 
for the regularized and the original system simultaneously.
Indeed, we use the same scheme to construct approximate solutions
for $\gamma=0$ and for $\gamma>0$.
On the one hand, 
the passage to the time-continuous limit 
only yields a weak solution in the case $\gamma>0$
due to the availability of stronger a priori estimates.
On the other hand, 
the existence of energy-variational solutions
can be shown under more general assumptions on $\calP$ and $\varphi$.
In particular, one can allow for non-quadratic $\varphi$
if $\gamma=0$.

\paragraph{Outline}
In Section~\ref{sec:modeling},
we further motivate the model~\eqref{eq:system},
and we derive a corresponding weak formulation.
In Section~\ref{sec:existence}
we list our main assumptions, 
introduce the notions of weak solutions to~\eqref{eq:system.reg}
and of energy-variational solutions to~\eqref{eq:system}.
We present the main existence results
and study the consistency of both concepts with classical solutions.
Section~\ref{sec:approximation} is dedicated to the
construction of approximate solutions,
while Section~\ref{sec:limit.passage} concerns the
passage to the limit,
which finally completes the existence proofs.

\paragraph{Notation}
For the Euclidean inner product between vectors, matrices or tensors, we write
\[
\ba\cdot\bb=\sum_{j=1}^3\ba_j\bb_j,
\quad
\bbA:\bbB=\sum_{j,k=1}^3\bbA_{jk}\bbB_{jk},
\quad
\bbG\tcolon\bbH=\sum_{j,k,\ell=1}^3\bbG_{jk\ell}\bbH_{jk\ell},
\]
for vectors $\ba=(\ba_j),\,\bb=(\bb_j)\in\R^3$, 
matrices $\bbA=(\bbA_{jk}),\,\bbB=(\bbB_{jk})\in\R^{3\times3}$,
and third-order tensors $\bbG=(\bbG_{jk\ell}),\,\bbH=(\bbH_{jk\ell})\in\R^{3\times3\times3}$.
Moreover, $\bbI\in\R^{3\times3}$ is the identity tensor,
and the trace of $\bbA$ is denoted by $\tr\bbA=\bbA:\bbI$.
By $\sph\bbA=\frac{1}{3}(\tr\bbA)\bbI$, $\dev\bbA=\bbA-\frac{1}{3}(\tr\bbA)\bbI$
and $\sym\bbA=\frac{1}{2}(\bbA+\bbA^\top)$,
we denote the spherical, deviatoric and symmetric part of $\bbA$,
respectively.
We introduce the subspaces of symmetric and symmetric deviatoric 
matrices by
\[
\begin{aligned}
\R^{3\times3}_{\sym}
&=\setcl{\bbA\in\R^{3\times3}}{\bbA^\top=\bbA},
\\
\R^{3\times3}_{\sym,0}
&=\setcl{\bbA\in\R^{3\times3}}{\bbA^\top=\bbA,\ \tr\bbA=0}.
\end{aligned}
\]

Let $X$ be a Banach space and $\calP\colon X\to[0,\infty]$ be a convex functional. 
By $\dom\calP=\setc{x\in X}{\calP(x)<\infty}$ we denote its domain,
$\calP^*$ denotes the convex conjugate and $\partial\calP$ 
is the convex subdifferential.
For a continuously differentiable function $\varphi\colon X\to\R$,
we denote its Fr\'echet derivative by $D\varphi$.

\section{Derivation of the model and weak formulation}
\label{sec:modeling}

To derive the model~\eqref{eq:system},
we explain the considered visco-elasto-plastic rheology
and the incorporation of the assumption of slight compressibility.
Subsequently, we derive the energy-dissipation balance,
and we motivate a weak formulation 
in terms of a variational inequality.

\subsection{Visco-elasto-plastic rheology}

The main motivation for the model~\eqref{eq:system} 
arises from geophysics,
where related models have been used to study
the rock deformation in the Earth's lithosphere~\cite{MoresiDufourMuehlhaus2002,GeryaYuen2007,Preussetal2019}.
In this context, the deformation is
considered on very long time scales
and with very slow speeds of only millimeters per year.
Due to the lack of a plausible reference configuration, 
an Eulerian description is preferred 
over a Lagrangian description~\cite{roubicek_timediscviscoel_2025}.
Therefore, we consider the continuity equation and the momentum equation
given by
\begin{align}
\partial_t\rho + \dv(\rho\bv)&=0,
\label{eq:NScomp.cont}
\\
\rho(\partial_t\bv + \bv\cdot\nabla\bv)&=\dv(\bbS(\bv)+\bbT)+\bf
\label{eq:NScomp.mom}
\end{align}
for the density $\rho$, the Eulerian velocity field $\bv$ and an external force $\bf$,
where the total stress is the sum of two parts $\bbS(\bv)$ and $\bbT$. 
Here, $\bbS(\bv)$ is a classical linear viscous stress tensor.
A standard choice for isotropic viscous behavior
would be the Newtonian stress tensor
\begin{equation}
\label{eq:Newtonian.stress}
\bbS(\bv)
= 2\mu_1 \Big( \bbD(\bv) - \frac13 \dv \bv \bbI \Big) + \mu_2 \dv \bv \bbI
=2\mu_1\dev\bbD(\bv)+3\mu_2\sph\bbD(\bv),
\end{equation}
where $\mu_1,\mu_2>0$
denote shear and bulk viscosity coefficients.

To describe the deformation of tectonic plates, the development of faults,
and the occurrence of aseismic creep,
such geophysical models combine viscous, elastic and plastic effects~\cite{MoresiDufourMuehlhaus2002,GeryaYuen2007,HerrenGeryaVand2017IRSDF, Preussetal2019, BabeykoSobolev2008,gerya_2019,PopovSobolev2008, Roubicek2021,Yarushina2015DecompactionOP},
which we incorporate in terms of the stress tensor $\bbT$.
We consider different rheologies for volumetric and the shear deformation.
While the spherical stress $\sph\bbT=\frac{1}{3}(\tr\bbT)\bbI$ is purely elastic,
the deviatoric part $\dev\bbT$ is subject to 
a Maxwell-type visco-elasto-plastic rheology.
We consider the case of small strain and
employ the Green--Naghdi-type additive decomposition into 
the elastic strain $\bbE$ and the inelastic strain $\bbP$.
For the strain rate $\bbD(\bv)=\frac{1}{2}(\nabla\bv+\nabla\bv^\top)$, 
this gives
$\bbD(\bv)=\jder{\bbE}+\jder{\bbP}$.
Here, the $\jder{(\cdot)}$ denotes an objective tensor rate,
which we discuss below, see~\eqref{eq:jaumannder}.

The elastic strain $\bbE$ is associated with the stress $\bbT$ 
through a stored energy potential $\varphi\colon\R^{3\times3}_{\sym}\to[0,\infty)$ 
such that
\begin{equation}
\label{eq:stress.strain.el}
\bbT=D\varphi(\bbE)+\varphi(\bbE)\bbI.
\end{equation}
Here the second term is an `internal' pressure that 
is due to the transport of the elastic energy.
By allowing for a non-quadratic stored energy $\varphi$, 
the model is also applicable for large displacements,
which is relevant in the geophysical context. 
Note that this does not contradict the above assumption of small strain,
compare~\cite{roubicek_timediscviscoel_2025}.

The inelastic strain $\bbP$ is assumed to be purely isochoric
and governed by the stress-strain relation 
$\dev\bbT\in\partial\calR(\jder{\bbP})$,
where $\calR$ is a convex dissipation potential,
which can be non-smooth in general.
Invoking the dual potential $\calP=\calR^*$,
we can thus reformulate the decomposition $\bbD(\bv)=\jder{\bbE}+\jder{\bbP}$ 
of the strain rate as
\begin{equation}
\label{eq:strain.dec}
\bbD(\bv)\in\jder{\bbE}+\partial\calP(\dev\bbT).
\end{equation}
One common characteristic of the models
from~\cite{MoresiDufourMuehlhaus2002,HerrenGeryaVand2017IRSDF, GeryaYuen2007,Preussetal2019}
is the transition from visco-elastic to plastic behavior in the deviatoric component 
when a certain yield stress $\sigma_{\mathrm{yield}}>0$ is reached.
This feature reflects the brittle nature of rock.
Mathematically, it can be incorporated in $\calP$ by the choice 
$\calP(\bbA)
=\int_\Omega P(\bbA(x))\dd x$, 
with $P(\bbA)=+\infty$ for $|\bbA|>\sigma_\mathrm{yield}$.
Combined with a linear viscosity law, this gives~\eqref{eq:disspot.example}.

As we consider a convective model, 
we also have to choose a suitable objective rate $\jder\bbE$.
Objectivity means that it is invariant under a (time-dependent) change of frame.
There are different choices available,
and we use Zaremba--Jaumann rate
\begin{equation}\label{eq:jaumannder}
\jder{\bbE}
=\partial_t\bbE + \bv\cdot\nabla\bbE + \bbE \bbW(\bv)-\bbW(\bv)\bbE,
\qquad
\bbW(\bv)
=\frac{1}{2}(\nabla\bv-\nabla\bv^\top),
\end{equation}
which is the common choice in the field of geophysics~\cite{MoresiDufourMuehlhaus2002,GeryaYuen2007,HerrenGeryaVand2017IRSDF, Preussetal2019, BabeykoSobolev2008,gerya_2019,PopovSobolev2008, Roubicek2021,Yarushina2015DecompactionOP}.
We further refer to~\cite[p.\,494]{Biot1965IncrementalDeformation} 
for a theoretical justification of the Zaremba--Jaumann rate,
and to \cite{XiaoBruhnsMeyers1998corotationalrates,MeyersSchiesseBruhns2000objectiverates}
for a more general overview over objective tensor rates.
One readily sees that~\eqref{eq:strain.dec} preserves symmetry of $\bbE$.
Moreover, since we have $\tr(\bbE \bbW(\bv)-\bbW(\bv)\bbE)=0$, 
we can decompose~\eqref{eq:strain.dec} as
\begin{equation}\label{eq:separation.strainrate}
\partial_t\tr(\bbE)+\bv\cdot\nabla\tr(\bbE)=\dv\bv, 
\qquad
(\dev\bbE)\jder{\,}+\partial \calP(\dev\bbT)\ni \bbD(\bv)-\frac{1}{3}(\dv\bv)\bbI,
\end{equation}
that is, the evolution of volumetric and shear strain remain separated.
Moreover, if $D\varphi(\bbE)$ and $\bbE$ commute, we have
\[
\lp{\bbE \bbW(\bv)-\bbW(\bv)\bbE}:D\varphi(\bbE)
=\bbW(\bv):\lp{\bbE^\top D\varphi(\bbE)-D\varphi(\bbE)\bbE^\top}=0
\]
where we used the symmetry of $\bbE$,
which yields
\begin{equation}
\label{eq:Edot.Dphi}
\jder{\bbE}:D\varphi(\bbE)=\partial_t\varphi(\bbE) + \bv\cdot\nabla \varphi(\bbE),
\end{equation}
so that Zaremba--Jaumann rate leads to a simple transport of the stored elastic energy.
Note that $D\varphi(\bbE)$ and $\bbE$ commute if and only if $\bbT$ and $\bbE$ commute,
which is true for isotropic and frame-indifferent materials,
see~\cite[Sect.\,37]{Gurtin1981IntoContMech} for instance.

The model consisting of \eqref{eq:NScomp.cont}, \eqref{eq:NScomp.mom},
\eqref{eq:stress.strain.el}, \eqref{eq:strain.dec}
is a nonlinear generalization 
the model from~\cite{MoresiDufourMuehlhaus2002,HerrenGeryaVand2017IRSDF, GeryaYuen2007,Preussetal2019},
where linear elasticity and viscosity are combined with 
perfect plasticity.
In this case, the stored energy potential is given by
\begin{equation}
\label{eq:storedenergy.example}
\varphi(\bbE)=\frac{K}{2}|\sph\bbE|^2+\frac{G}{2}|\dev\bbE|^2
\end{equation}
with elastic bulk modulus $K>0$ and elastic shear modulus $G>0$,
and the (dual) dissipation potential $\calP$ is given by~\eqref{eq:disspot.example}.
Moreover, the present model generalizes the large-deformation 
small-strain viscoelastic model from~\cite{roubicek_timediscviscoel_2025} 
by including plastic effects.

In summary,
we consider a Kelvin--Voigt rheology
in the volumetric/spherical part,
combined with a Jeffreys (or anti-Zener) rheology
enriched with plasticity
in the isochoric/deviatoric part.
In particular, the model allows for isochoric creep
and isochoric plastic deformation.

\subsection{Slightly compressible approximation}

In many situations,
like the flow of water or the deformation of rock,
the material compression is negligible compared to other effects during the evolution.
On the one hand, this motivates to simplify the model
by directly imposing incompressibility as an assumption, 
and to work with the Navier--Stokes equations 
for incompressible flow.
Such models were analytically studied
in~\cite{EiterHopfMielke2022,EiterHopfLasarzik2023,ChengLasarzikThomas2025CahnHillardVisco}.
On the other hand, such a restriction to divergence-free velocity fields
would not allow for the propagation of pressure waves,
corresponding to sound waves or longitudinal seismic waves.

Therefore, we consider a model for slight compressibility,
which is an approximation of~\eqref{eq:NScomp.cont},~\eqref{eq:NScomp.mom} that 
assumes a constant density
without restricting to the divergence-free velocities.
Clearly, this contradicts the continuity equation~\eqref{eq:NScomp.cont},
and simply omitting it would destroy the energetic structure of the evolution.
In particular, the kinetic energy-dissipation balance
\begin{equation}\label{eq:kinetic.diss}
\partial_t\Lp{\frac{\rho|\bv|^2}{2}}+\dv\Lp{\frac{\rho|\bv|^2}{2}\bv}
+\Sigma:\bbD(\bv)-\dv(\Sigma^\top\bv)=\bf\cdot\bv, 
\qquad \Sigma=\bbS(\bv)+\bbT,
\end{equation}
would no longer hold.
To compensate this, one introduces an additional internal force in the momentum balance~\eqref{eq:NScomp.mom}.
More precisely, for a homogeneous material with constant density $\rho$,
we consider
\begin{equation}\label{eq:mom.semicomp}
\rho(\partial_t\bv+\bv\cdot\nabla \bv)
=\dv (\bbS(\bv)+\bbT)-\frac{\rho}{2}(\dv \bv)\bv+\bf.
\end{equation}
This approach goes back to Temam~\cite{Temam1968,Temam1969},
see also~\cite[Ch.\,III, Sect.\,8]{Temam1977NSE}.
The new term ensures that the kinetic energy
is transported properly,
and that multiplying~\eqref{eq:mom.semicomp} with $\bv$
leads to the energy-dissipation balance~\eqref{eq:kinetic.diss}.
Note that there exist other possible modifications
of the momentum balance
that preserve~\eqref{eq:kinetic.diss},
compare~\cite[Remark 1]{roubicek_quasisemicompressible_2021}.
The choice~\eqref{eq:mom.semicomp}
has the advantage 
that one arrives at the model for incompressible flow by setting $\dv\bv=0$.
In several works, a combination of~\eqref{eq:mom.semicomp} 
with an evolution equation for the pressure
is used to approximate the incompressible Navier--Stokes flow
numerically or
analytically~\cite{ChenLaytonMcLaughlin2019,DeCariaLaytonMcLaughlin2017, Prohl1997,DonatelliMarcati2010,DonatelliMarcati2006,BerselliSpirito2018}.
For further discussion and references of such slightly compressible approximations, 
we refer to the recent article~\cite{roubicek_quasisemicompressible_2021}.

Combining the adapted momentum equation~\eqref{eq:mom.semicomp} 
with the visco-elasto-plastic stress-strain relations~\eqref{eq:stress.strain.el} and~\eqref{eq:strain.dec}
and suitable initial and boundary conditions,
we finally arrive at the full continuum model~\eqref{eq:system}.
Observe that~\eqref{eq:system} generalizes the model for \textit{semi-compressible} viscous flow from~\cite{roubicek_quasisemicompressible_2021},
where a pure spherical stress $\bbT=-(\pi+\frac{\pi^2}{2K})\bbI$ with elastic bulk modulus $K>0$ and pressure $\pi$ is considered, which are subject to 
\[
\frac{1}{K}\partial_t\pi + \frac{1}{K}\bv\cdot\nabla \pi +\dv\bv=0.
\]
Indeed, this particular case can be obtained from~\eqref{eq:mom.semicomp},~\eqref{eq:stress.strain.el},~\eqref{eq:strain.dec}
by choosing $\varphi$ as in~\eqref{eq:separation.strainrate} with $G=0$
and by setting $\pi=-\frac{1}{3}\tr(\bbT)$.

\subsection{Energy dissipation}
\label{sec:energetics}

To get further insight into the energetics of the system~\eqref{eq:system},
we consider $\Xi\in\partial\calP(\dev\bbT)$ and multiply~\eqref{eq:sys.strain}
with $D\varphi(\bbE)$.
Using~\eqref{eq:Edot.Dphi} and integration by parts,
we calculate
\[
\begin{aligned}
0=\lp{\jder{\bbE}+\Xi-\bbD(\bv)}:D\varphi(\bbE)
&=\partial_t\varphi(\bbE)
+ \bv\cdot\nabla\varphi(\bbE)+\Xi:D\varphi(\bbE)-\bbD(\bv):D\varphi(\bbE)
\\
&=\partial_t\varphi(\bbE)+\dv\lp{\varphi(\bbE)\bv}
+\Xi:\bbT-\bbD(\bv):\bbT,
\end{aligned}
\]
where we used~\eqref{eq:sys.stress} and that $\tr\Xi=0$.
Due to the boundary conditions~\eqref{eq:dirichletbdry} and~\eqref{eq:neumannbdry},
combining this identity with~\eqref{eq:kinetic.diss}
and integrating over $\Omega$,
we conclude
\begin{equation}
\label{eq:eneq.derived}
\frac{\mathrm d}{\mathrm d t}\int_\Omega \frac{\rho|\bv|^2}{2}+\varphi(\bbE)\dd x
+\int_\Omega \bbS(\bv):\bbD(\bv)+\Xi:\bbT \dd x 
=\int_\Omega\bf\cdot\bv\dd x
\end{equation}
for $\Xi\in\partial\calP(\dev\bbT)$.
This shows that the total energy 
is given by
\begin{equation}
\label{eq:energy.def}
\calE(\bv,\bbE)
:=\int_\Omega\frac{\rho}{2}|\bv|^2+\varphi(\bbE)\dd x
\end{equation}
as the sum of the kinetic energy and the stored elastic energy.
It is dissipated 
by the Newtonian viscous stress $\bbS(\bv)$
and the deviatoric stress associated with the (dual) dissipation potential $\calP$,
and it is subject to the power of the external force $\bf$.

As $\Xi\in\partial P(\dev\bbT)$, we have $\Xi=\dev\Xi$, 
and we can use the Fenchel equivalence
to write $\int_\Omega\Xi:\bbT \dd x=\calP(\dev\bbT)+\calP^*(\Xi)$.
Employing this identity in~\eqref{eq:eneq.derived}
and omitting $\calP^*(\Xi)\geq 0$, which follows from $\calP(0)=0$,
leads to the energy-dissipation inequality
\begin{equation}
\label{eq:enin.P}
\frac{\mathrm d}{\mathrm d t}\calE(\bv,\bbE)
+\int_\Omega \bbS(\bv):\bbD(\bv) \dd x +\calP(\dev\bbT)
\leq\int_\Omega\bf\cdot\bv\dd x.
\end{equation}
Although this inequality is not equivalent to~\eqref{eq:eneq.derived},
it has the advantage that the subdifferential element
$\Xi\in\partial P(\dev\bbT)$ is not present in~\eqref{eq:enin.P}.
Moreover, a weak form of~\eqref{eq:enin.P}
is directly included in the generalized solution concepts discussed in this article.

\subsection{Towards a weak formulation}
\label{subsec:weakform}

To derive a weak formulation of~\eqref{eq:system},
we follow the approach from~\cite{roubicek2005nonlPDEs,EiterHopfMielke2022,EiterHopfLasarzik2023}.
Instead of simply testing~\eqref{eq:sys.strain} and integrating,
we express the subdifferential in terms of the inequality
\[
\Xi\in\partial\calP(\dev\bbT)\iff 
\forall \Psi: \ \calP(\dev\bbT)-\calP(\dev\Psi)
\leq \int_\Omega \Xi:(D\varphi(\bbE)-\Psi)\dd x,
\]
where $\Psi$ belongs to a suitable class of $\R^{3\times3}_{\sym}$-valued 
test functions.
Here we used $\dev\bbT=\dev D\varphi(\bbE)$ 
and that $\calP$ is defined on deviatoric tensors,
so that $\tr\Xi=0$.
With this characterization,
we can express~\eqref{eq:sys.strain} as
\[
\int_\Omega \jder{\bbE}{\,:\,} D\varphi(\bbE) 
- \bbD(\bv) {\,:\,}D\varphi(\bbE) \dd x
+ \calP(\dev D\varphi(\bbE))
-\int_\Omega \jder{\bbE}{\,:\,}\Psi-\bbD(\bv){\,:\,}\Psi \dd x
-\calP(\dev\Psi)
\leq 0.
\]
This corresponds to formally testing~\eqref{eq:strain.dec} with $D\varphi(\bbE)-\Psi$.
We proceed similarly with the momentum equation~\eqref{eq:sys.mom}
and multiply it with $\bv-\psi$ for a test function $\psi$ and integrate.
Adding up both equations, integrating over a time interval $(s,t)$ for $s<t$
and using integration by parts, 
where we make use of the boundary conditions~\eqref{eq:dirichletbdry} and~\eqref{eq:neumannbdry},
we arrive at
\begin{equation}
\label{eq:weakform.original}
\begin{aligned}
&\Lb{\calE(\bv,\bbE)-\int_\Omega\rho\bv\cdot\psi+\bbE:\Psi\dd x}\Big|_s^t
+\int_s^t\int_\Omega \bbS(\bv):\bbD(\bv)-\bf\cdot\bv \dd x \dd t'
\\
&\quad
+\int_s^t\int_\Omega
\rho\bv\cdot\partial_t\psi
+\frac{\rho}{2}(\dv\bv)(\bv\cdot\psi)
+\rho (\bv\otimes\bv):\nabla\psi \dd x \dd t'
\\
&\quad+\int_s^t\int_\Omega
- \lp{\bbS(\bv)+D\varphi(\bbE)}:\bbD(\psi)
-\varphi(\bbE)\dv\psi 
+\bf\cdot\psi\dd x \dd t'
\\
&\quad
+\int_s^t\int_\Omega
\bbE:\partial_t\Psi
+(\dv\bv)\bbE:\Psi+\bv\cdot\nabla \Psi:\bbE - (\bbE\bbW(\bv)-\bbW(\bv)\bbE):\Psi
 \dd x \dd t'
\\
&\quad+\int_s^t\int_\Omega
\bbD(\bv):\Psi  \dd x \dd t'
+\int_s^t\calP(\dev D\varphi(\bbE)) -\calP(\dev\Psi)\dd t'\leq 0
\end{aligned}
\end{equation}
for all test functions $\psi\in C_c^1([0,T)\times\ol\Omega;\R^3)$
and $\Psi\in C_c^1([0,T)\times\ol\Omega;\R_{\sym}^{3\times 3})$
with $\psi=0$ on $(0,T)\times\Gamma_{\rmD}$ and 
$\psi\cdot\bn=0$ on $(0,T)\times\Gamma_{\rmN}$.
The calculations are similar to the derivation of~\eqref{eq:eneq.derived},
where we tested~\eqref{eq:sys.mom} and~\eqref{eq:sys.strain}
with $\bv$ and $D\varphi(\bbE)$, respectively.
If we ignore $\calP$, then~\eqref{eq:weakform.original}
is merely the sum of an energy inequality 
and the standard weak formulations of~\eqref{eq:sys.mom} and~\eqref{eq:sys.strain}.
Observe that for $(\psi,\Psi)=0$, 
we rediscover a time-integrated version of the energy-inequality~\eqref{eq:enin.P}.
Moreover, the initial conditions~\eqref{eq:initial} are satisfied in a weak sense
if~\eqref{eq:weakform.original} holds for $s=0$ 
with $\bv(0)=\bv_0$ and $\bbE(0)=\bbE_0$.

For a reformulation of~\eqref{eq:weakform.original},
we can make use of the following result.

\begin{lemma}\label{lem:ineq.pointwisevariational}
Let $f\in L^1(0,T)$, $g\in L^\infty(0,T)$ and $g_0\in\R$.
Then the inequality 
\begin{equation}
-\int_0^T \phi'(\tau) g(\tau) \dd \tau  + \int_0^T \phi(\tau) f(\tau) \dd \tau - \phi(0)g_0 \leq 0 
\label{ineq1}
\end{equation}
holds for all $\phi \in C^1_c ([0,T))$ with $\phi \geq 0$
if and only if the inequality
\begin{equation}
    g(t) -g(s) + \int_s^t f(\tau) \dd \tau \leq 0 
    \label{ineq2}
\end{equation}
holds for a.e.~$s,\, t\in[0,T)$ with $s<t$,
including $s=0$ if we replace $g(0)$ with $g_0$.
In these cases, $g$ coincides a.e.~with an element of $\BV([0,T])$ 
with $\lim_{s\searrow0}g(s)\leq g_0$ and $\lim_{s\searrow t} g(s)\leq \lim_{s\nearrow t} g(s)$ for all $t\in(0,T)$.
\end{lemma}

\begin{proof}
See~\cite[Lemma 2.5]{EitLas24envarhyp}.
\end{proof}

\begin{remark}\label{rem:weakform.variational}
The formulation~\eqref{eq:weakform.original}
is the basis of the solution concepts introduced 
in Section~\ref{sec:existence}.
To avoid the pointwise evaluation at time $s$ and $t$, 
we can also use a time-variational form
that is more accessible by convergence arguments later. 
By Lemma~\ref{lem:ineq.pointwisevariational},
the inequality~\eqref{eq:weakform.original} holds for a.a.~$s<t$ if and only if
\begin{equation}
\label{eq:weakform.original.variational}
\begin{aligned}
&\int_0^T-\partial_t\phi\Lb{\calE(\bv,\bbE)-\int_\Omega\rho\bv\cdot\psi+\bbE{\,:\,}\Psi\dd x}\dd t
+\int_0^T\phi\int_\Omega \bbS(\bv){\,:\,}\bbD(\bv)-\bf\cdot\bv \dd x \dd t
\\
&\quad
+\int_0^T\phi\int_\Omega
\rho\bv\cdot\partial_t\psi
+\frac{\rho}{2}(\dv\bv)(\bv\cdot\psi)
+\rho (\bv\otimes\bv):\nabla\psi \dd x \dd t
\\
&\quad+\int_0^T\phi\int_\Omega
- \lp{\bbS(\bv)+D\varphi(\bbE)}:\bbD(\psi)
-\varphi(\bbE)\dv\psi 
+\bf\cdot\psi\dd x \dd t
\\
&\quad
+\int_0^T\phi\int_\Omega
\bbE:\partial_t\Psi
+(\dv\bv)\bbE:\Psi+\bv\cdot\nabla \Psi:\bbE - (\bbE\bbW(\bv)-\bbW(\bv)\bbE):\Psi
 \dd x \dd t
\\
&\quad+\int_0^T\phi(t)\int_\Omega
\bbD(\bv):\Psi  \dd x \dd t'
+\int_0^T\phi\lp{\calP(\dev D\varphi(\bbE)) -\calP(\dev\Psi)}\dd t
\\
&\quad-\phi(0)\Lb{\calE(\bv_0,\bbE_0)-\int_\Omega\rho\bv_0\cdot\psi(0)+\bbE_0:\Psi(0)\dd x}\leq 0
\end{aligned}
\end{equation}
holds for all $\phi \in C^1_c ([0,T))$ with $\phi \geq 0$.
We further obtain
that
\[
\begin{aligned}
\lim_{s\nearrow 0}\Lb{\calE(\bv(s),\bbE(s))-\int_\Omega\rho\bv(s)\cdot\psi+\bbE(s){\,:\,}\Psi\dd x}
&=\Lb{\calE(\bv_0,\bbE_0)-\int_\Omega\rho\bv_0\cdot\psi+\bbE_0{\,:\,}\Psi\dd x}
\end{aligned}
\]
for all $\psi\in C_c^1(\ol\Omega;\R^3)$
and $\Psi\in C_c^1(\ol\Omega;\R_{\sym}^{3\times 3})$
with $\psi=0$ on $(0,T)\times\Gamma_{\rmD}$ and 
$\psi\cdot\bn=0$ on $(0,T)\times\Gamma_{\rmN}$.
Varying $\psi$ and $\Psi$, 
a standard variational argument yields
that $\calE(\bv(s),\bbE(s))\to\calE(\bv_0,\bbE_0)$ in $\R$ and
$(\bv(s),\bbE(s))\tows(\bv_0,\bbE_0)$ weakly*
in the dual space of the test functions. 
In this sense, the initial condition~\eqref{eq:initial} is satisfied.
A similar argument shows that $(\bv,\bbE)$ is weaky* continuous
in this functional framework.
\end{remark}

\section{The generalized solution concepts}
\label{sec:existence}

We first collect the main assumptions 
to develop the theory.
Subsequently, we introduce the concepts of 
weak solutions to~\eqref{eq:system.reg}
and of energy-variational solutions to~\eqref{eq:system},
and we state the main results on 
the global-in-time existence of such solutions.

\subsection{General assumptions}
To perform a rigorous existence analysis, we make the following assumptions:
\begin{enumerate}[label=(A\arabic*)]
\item
\label{ass:Omega}
Let $\Omega\subset\R^3$ be a domain with Lipschitz boundary $\partial\Omega$, 
and let $\Gamma_D\subset\partial\Omega$ be measurable.
If $\Gamma_D$ has vanishing surface measure, assume in addition that 
there is no rigid motion that leaves $\Omega$ invariant.
\item
\label{ass:f}
Let $\bf\in L^1(0,T;L^2(\Omega;\R^3))$.
\item
\label{ass:viscstress}
Let the linear viscous stress $\bbS(\bv)$ be given by $\bbS(\bv)=\bbL[\bbD(\bv)]$
for a linear mapping $\bbL\colon\R^{3\times3}_{\sym}\to\R^{3\times3}_{\sym}$
such that there exists $\mu_*>0$ with
$\bbL[\bbA]:\bbA\geq \mu_*|\bbA|^2$.
\item
\label{ass:P}
Let $\calP\colon L^1(\Omega;\R^{3\times3}_{\sym,0})\to[0,\infty]$ be nonnegative,
convex and lower semicontinuous, and let $\calP(0)=0$.
\item
\label{ass:phi}
Let $\varphi\in C^3(\R^{3\times 3}_{\sym};[0,\infty))$ be a nonnegative, convex function
such that 
\begin{enumerate}[label=(\roman*)]
\item
\label{ass:phi.strconv}
$\varphi$ is strongly convex, that is, there exists $\kappa>0$ such that $\varphi-\frac{\kappa}{2}|\cdot|^2$ is convex,
\item
\label{ass:phi.D3phibounded}
there exists $M>0$ such that $|D^3\varphi(\bbE)|\leq M$ for all $\bbE\in\R^{3\times 3}_{\sym}$,
\item
\label{ass:phi.commute}
it holds $D\varphi(\bbE)\,\bbE=\bbE\,D\varphi(\bbE)$ for all $\bbE\in\R^{3\times3}_{\sym}$.
\end{enumerate}
\item
\label{ass:PDphi}
Let the function $\calP\circ [\dev D\varphi]\colon 
L^2(\Omega;\R^{3\times 3}_{\sym}) \to [0,\infty]$,
given by
$\bbE\mapsto\calP(\dev D\varphi(\bbE))$, be convex and lower semicontinuous.
\end{enumerate}

While these assumptions are sufficient to show the existence of energy-variational solutions to the original system~\eqref{eq:system},
we need to restrict to quadratic stored energy potentials $\varphi$
and to require an additional property of $\calP$ 
to obtain weak solutions for the regularized system~\eqref{eq:system.reg}:
\begin{enumerate}[label=(A\arabic*')]
\setcounter{enumi}{4}
\item
\label{ass:phi.quadratic}
Let $\varphi$ be quadratic and strictly convex, that is, 
$\varphi(\bbE)=\frac{1}{2}\bbK[\bbE]\mathbin:\bbE$
for some linear function
$\bbK\colon\R^{3\times3}_{\sym}\to\R^{3\times3}_{\sym}$ that satisfies
$\bbK[\bbA]\mathbin:\bbB=\bbK[\bbB]\mathbin{:}\bbA$ and $\bbK[\bbA]\mathbin:\bbA\geq\kappa \abs{\bbA}^2$
for all $\bbA,\bbB\in\R^{3\times3}_{\sym}$ and some $\kappa>0$.
\item
\label{ass:P.gamma}
There exists $R_0>0$ such that $\calP(\Psi)<\infty$
for all $\Psi\in C^1(\overline{\Omega};\R^{3\times3}_{\sym})$ such that $\norm{\Psi}_{C^1(\overline{\Omega};\R^{3\times3}_{\sym})}\leq R_0$.
\end{enumerate}

The additional geometric condition in~\ref{ass:Omega}
has to be understood in combination with~\ref{ass:viscstress}.
Both assumptions together imply the Korn-type inequality
\begin{equation}\label{eq:Korn}
\int_\Omega \bbS(\bv):\bbD(\bv) \dd x 
\geq  \mu \int_\Omega |\nabla\bv|^2\dd x
\end{equation}
for some constant $\mu>0$ and
all $\bv\in H^1(\Omega;\R^3)$
with $\bv=0$ on $\Gamma_{\rmD}$ and $\bv\cdot\bn=0$ on $\partial\Omega$.
Indeed, by~\ref{ass:viscstress} we ensure that
\[
\bbS(\bv):\bbD(\bv)=\bbL[\bbD(\bv)]:\bbD(\bv)\geq \mu_* |\bbD(\bv)|^2,
\]
which allows to conclude~\eqref{eq:Korn} by a classical compactness argument.
However, to this end,
we have to exclude that $\bv$ is a rigid motion, 
which results from the boundary conditions if~\ref{ass:Omega} holds.

Observe that the strong convexity of $\varphi$ from~\ref{ass:phi}\ref{ass:phi.strconv}
implies that 
$\varphi$ has at least quadratic growth, that is,
there exists $\beta>0$ such that for all $\bbE\in\R^{3\times 3}_{\sym}$ it holds
\begin{equation}\label{eq:phi.quadrgrowth}
1+\varphi(\bbE)\geq \beta|\bbE|^2,
\end{equation}
which follows from Taylor's theorem.

The bound on $D^3\varphi$ in~\ref{ass:phi}\ref{ass:phi.D3phibounded}
will be used at several points in the argument.
Firstly, it ensures that $D\varphi(\bbE)\in L^1(\Omega;\R^{3\times3})$
for $\bbE\in L^2(\Omega;\R^{3\times3})$ by Taylor's theorem.
Therefore, the term $D\varphi(\bbE):\bbD(\psi)$ in the weak formulation is integrable,
and this explains the domain of $\calP$ in~\ref{ass:P}.
Moreover,
we can `convexify' this term 
with the help of the energy, 
which will be crucial for the proposed construction of approximate solutions
and for the derivation of energy-variational solutions,
see Lemma~\ref{lem:Kconvexifies} and Lemma~\ref{lem:K.giveslsc} below.

Assumption~\ref{ass:phi}\ref{ass:phi.commute} means that the elastic strain $\bbE$ commutes with the stress $\bbT=D\varphi(\bbE)+\varphi(\bbE)\bbI$. 
In particular, this is the case for an isotropic material,
which comes along with $\varphi(\bbQ\bbE\bbQ^\top)=\varphi(\bbE)$ 
for all rotational matrices 
$\bbQ\in \mathrm{SO}(3)$.
Indeed, $\varphi$ can be expressed as a function of the principal
invariants of $\bbE$ in this case, so that 
$D\varphi(\bbE)=a\bbI+b\bbE+c\bbE^2$ for scalars $a,b,c$
that are functions of the principal invariants of $\bbE$,
compare~\cite[Sect.\,37]{Gurtin1981IntoContMech}.

Clearly, the potential from~\eqref{eq:disspot.example}
modeling visco-plasticity satisfies~\ref{ass:P}.
Note that 
we do not impose any growth assumptions on $\calP$ in~\ref{ass:P},
and we shall not use $\calP$ to derive additional bounds on the solutions.
Therefore, even the choice $\calP=0$ is possible,
which would lead to pure elastic strain
and reduce~\eqref{eq:system} to a model of Kelvin--Voigt type.

Assumption~\ref{ass:PDphi} concerns the interplay between the 
energy potential $\varphi$ and the dissipation potential $\calP$.
Notice that in the case of a quadratic stored elastic energy $\varphi$ as in~\ref{ass:phi.quadratic},
which satisfies~\ref{ass:phi} trivially,
assumption~\ref{ass:PDphi} is a direct consequence of~\ref{ass:P}.
The most relevant example for such a quadratic $\varphi$ is 
given in~\eqref{eq:storedenergy.example},
which allows for different elastic shear and bulk modulus.
We shall only use the strict assumption~\ref{ass:phi.quadratic}
if $\gamma>0$, 
where it renders the stress diffusion a linear term.

The technical assumption~\ref{ass:P.gamma} is used
to derive a suitable incremental bound for the elastic strain $\bbE$
in the case $\gamma>0$, see Lemma~\ref{lem:increment.est.E} below. 
This estimate will lead to strong convergence of the approximate sequence for $\bbE$,
which allows to pass to the limit in a weak formulation.
As strong convergence cannot be expected for $\gamma=0$,
this estimate is not necessary in this case,
and the limit object is identified as an energy-variational solution.

\subsection{Weak solutions to the system with stress diffusion}
\label{subsec:weaksol}

We introduce the notion of weak solutions to the regularized system~\eqref{eq:system.reg}
with $\gamma>0$.

\begin{definition}
\label{def:weaksol.gamma}
Let $\bv_0\in L^1(\Omega;\R^3)$ and $\bbE_0\in L^1(\Omega;\R^{3\times3})$  
such that $\calE(\bv_0,\bbE_0)<\infty$.
We call $(\bv,\bbE)$ a \textit{weak solution} to~\eqref{eq:system.reg}
if 
\[
\begin{aligned}
&\bv\in L^\infty(0,T;L^2(\Omega;\R^3))\cap L^2(0,T;H^1(\Omega;\R^3)),
\\
&\bbE\in L^\infty(0,T;L^2(\Omega;\R^{3\times 3}_{\sym}))\cap L^2(0,T;H^1(\Omega;\R^{3\times 3}_{\sym})),
\\
&\varphi(\bbE)\in L^\infty(0,T;L^1(\Omega)),
\quad
D\varphi(\bbE)\in L^2(0,T;H^1(\Omega;\R^{3\times 3}_{\sym})),
\end{aligned}
\]
if $\bv=0$ on $(0,T)\times\Gamma_{\rmD}$, 
if $\bv\cdot \bn=0$ on $(0,T)\times\Gamma_{\rmN}$, 
and if it holds
\begin{equation}
\label{eq:weakform.reg}
\begin{aligned}
&\Lb{\calE(\bv,\bbE)-\int_\Omega\rho\bv\cdot\psi+\bbE:\Psi\dd x}\Big|_s^t
+\int_s^t\int_\Omega \bbS(\bv):\bbD(\bv)+\gamma\absl{\nabla [D\varphi(\bbE)]}^2-\bf\cdot\bv \dd x \dd t'
\\
&\quad
+\int_s^t\int_\Omega
\rho\bv\cdot\partial_t\psi
+\frac{\rho}{2}(\dv\bv)(\bv\cdot\psi)
+\rho (\bv\otimes\bv):\nabla\psi \dd x \dd t'
\\
&\quad+\int_s^t\int_\Omega
- \lp{\bbS(\bv)+D\varphi(\bbE)}:\bbD(\psi)
-\varphi(\bbE)\dv\psi 
+\bf\cdot\psi\dd x \dd t'
\\
&\quad
+\int_s^t\int_\Omega
\bbE:\partial_t\Psi
+(\dv\bv)\bbE:\Psi+\bv\cdot\nabla \Psi:\bbE - (\bbE\bbW(\bv)-\bbW(\bv)\bbE):\Psi
 \dd x \dd t'
\\
&\quad+\int_s^t\int_\Omega
\bbD(\bv):\Psi - \gamma \nabla[D\varphi(\bbE)]\tcolon\nabla\Psi  \dd x \dd t'
+\int_s^t\calP(\dev D\varphi(\bbE)) -\calP(\dev\Psi)\dd t'\leq 0
\end{aligned}
\end{equation}
for a.a.~$s,t\in[0,T]$ with $s<t$,
and all $\psi\in C_c^1([0,T)\times\ol\Omega;\R^3)$
and $\Psi\in C_c^1([0,T)\times\ol\Omega;\R_{\sym}^{3\times 3})$
such that $\psi=0$ on $(0,T)\times\Gamma_{\rmD}$ and 
$\psi\cdot\bn=0$ on $(0,T)\times\Gamma_{\rmN}$.
The initial values $(\bv_0,\bbE_0)$ are attained in the sense that~\eqref{eq:weakform.reg}
is satisfied for $s=0$ and a.e.~$t\in[0,T]$ with $\bv(0)=\bv_0$ and $\bbE(0)=\bbE_0$.
\end{definition}

By setting $(\psi,\Psi)=0$ and $s=0$ in~\eqref{eq:weakform.reg},
we see that weak solutions to~\eqref{eq:system.reg} satisfy
the energy-dissipation inequality
\begin{equation}
\label{eq:enin.reg}
\begin{aligned}
\calE(\bv(t),\bbE(t))
+\int_0^t\int_\Omega \bbS(\bv):\bbD(\bv)+\gamma\absl{\nabla[D\varphi(\bbE)]}^2-\bf\cdot\bv \dd x \dd t'
\quad&
\\
+\int_0^t\calP(\dev D\varphi(\bbE))\dd t'&
\leq \calE(\bv_0,\bbE_0).
\end{aligned}
\end{equation}

Moreover, it follows from the definition
and Lemma~\ref{lem:ineq.pointwisevariational}
that any weak solution $(\bv,\bbE)$ is weakly* continuous 
in the dual of the test-function space,
and it attains the initial values in this sense,
compare Remark~\ref{rem:weakform.variational}.

\begin{remark}
As the test function $\psi$ appears in~\eqref{eq:weakform.reg}
in a linear way, 
we can use a standard variational argument to retain the weak form
of the momentum equation. 
Indeed, replacing $\psi$ with $\pm\lambda\psi$ in~\eqref{eq:weakform.reg},
dividing by $\lambda>0$, and passing to the limit $\lambda\to\infty$,
we obtain
\begin{equation}
\label{eq:weakform.momentum}
\begin{aligned}
&\Lb{-\int_\Omega\rho\bv\cdot\psi\dd x}\Big|_s^t
+\int_s^t\int_\Omega
\rho\bv\cdot\partial_t\psi
+\frac{\rho}{2}(\dv\bv)(\bv\cdot\psi)
+\rho (\bv\otimes\bv):\nabla\psi \dd x \dd t'
\\
&\quad+\int_s^t\int_\Omega
- \lp{\bbS(\bv)+D\varphi(\bbE)}:\bbD(\psi)
-\varphi(\bbE)\dv\psi 
+\bf\cdot\psi\dd x \dd t'=0.
\end{aligned}
\end{equation}
for a.a.~$s<t$ and all $\psi$ as in Definition~\ref{def:weaksol.gamma}.
Subtracting this identity from~\eqref{eq:weakform.reg},
we obtain a combination of an energy-dissipation inequality with 
the weak form of the regularized strain-rate equation~\eqref{eq:sys.strain.reg}.
In contrast to $\psi$, the function $\Psi$ appears in~\eqref{eq:weakform.reg}
in a nonlinear way due to the potential $\calP$.
Therefore, the previous argument cannot be applied 
and the strain-rate equation cannot be separated
in general.
Clearly, this is not an issue if $\calP=0$,
which would allow two separate both momentum and strain-rate equation from the energy-dissipation inequality.
\end{remark}

The introduction of diffusion in the stress-strain relation~\eqref{eq:sys.strain.reg} 
regularizes the system~\eqref{eq:system.reg}
and allows us to show the existence of global-in-time weak solutions.
Observe that the proposed method only allows to consider quadratic dissipation potentials $\varphi$
as in assumption~\ref{ass:phi.quadratic},
which yields $D\varphi(\bbE)=\bbK[\bbE]$ and
thus renders the stress-diffusion terms in~\eqref{eq:enin.reg} convex.

\begin{theorem}
\label{thm:existence.gamma}
Let $\gamma>0$, and let \ref{ass:Omega}--\ref{ass:P},~\ref{ass:phi.quadratic}, \ref{ass:P.gamma}
be satisfied.
For every $\bv_0\in L^1(\Omega;\R^3)$ and $\bbE_0\in L^1(\Omega;\R^{3\times3}_{\sym})$  
such that $\calE(\bv_0,\bbE_0)<\infty$,
there exists a weak solution $(\bv,\bbE)$ to~\eqref{eq:system.reg}
in the sense of Defintion~\ref{def:weaksol.gamma}.
\end{theorem}

For the proof of Theorem~\ref{thm:existence.gamma}, we refer to Subsect.\,\ref{subsec:limit.weak}.

\subsection{Energy-variational solutions to the original system}
\label{subsec:envarsol}

For the original system~\eqref{eq:system},
the existence of weak solutions in the sense of Definition~\ref{def:weaksol.gamma}
cannot be expected. 
Instead, we study energy-variational solutions.
For their definition, we do not follow~\cite{EiterHopfLasarzik2023},
but we adapt the refined energy-variational formulation 
from~\cite{EitLas24envarhyp,ALR24envarvisc,EiterLasarzikSliwinski_exselenvar}.

\begin{definition}
\label{def:envar}
Let $\bv_0\in L^1(\Omega;\R^3)$ and $\bbE_0\in L^1(\Omega;\R^{3\times3})$  
such that $\calE(\bv_0,\bbE_0)<\infty$.
We call a triple $(\bv,\bbE,E)$ 
\textit{an energy-variational solution}
to~\eqref{eq:system}
if 
\[
\begin{aligned}
&\bv\in L^\infty(0,T;L^2(\Omega;\R^3))\cap L^2(0,T;H^1(\Omega;\R^3)),
\\
&\bbE\in L^\infty(0,T;L^2(\Omega;\R^{3\times 3}_{\sym})),
\\
&E\in \BV([0,T]),
\end{aligned}
\]
if $\bv=0$ on $(0,T)\times\Gamma_{\rmD}$, 
if $\bv\cdot \bn=0$ on $(0,T)\times\Gamma_{\rmN}$, 
if $\calE(\bv,\bbE)\leq E$ a.e.~in $(0,T)$,
and if there exists a function $\calK\colon C^1(\ol\Omega;\R^3\times\R^{3\times3}_{\sym})\to[0,\infty)$
such that
\begin{equation}
\label{eq:envarform}
\begin{aligned}
&\Lb{E-\int_\Omega\rho\bv\cdot\psi+\bbE:\Psi\dd x}\Big|_s^t
+\int_s^t\int_\Omega \bbS(\bv):\bbD(\bv)-\bf\cdot\bv
\dd x\dd t'
\\
&\quad
+\int_s^t\int_\Omega
\rho\bv\cdot\partial_t\psi
+\frac{\rho}{2}(\dv\bv)(\bv\cdot\psi)+\rho (\bv\otimes\bv):\nabla\psi \dd x \dd t'
\\
&\quad
+\int_s^t\int_\Omega
- \lp{\bbS(\bv)+D\varphi(\bbE)}:\bbD(\psi)
-\varphi(\bbE)\dv\psi +\bf\cdot\psi \dd x \dd t'
\\
&\quad
+\int_s^t\int_\Omega
\bbE:\partial_t\Psi
+(\dv\bv)\bbE:\Psi+\bv\cdot\nabla \Psi:\bbE - (\bbE\bbW(\bv)-\bbW(\bv)\bbE):\Psi
 \dd x \dd t'
\\
&\quad+\int_s^t\int_\Omega
\bbD(\bv):\Psi \dd x \dd t'
+\int_s^t \calP(\dev D\varphi(\bbE))-\calP(\dev\Psi) \dd t'
\\
&\ 
\leq \int_s^t \mathcal K(\psi,\Psi)\lp{E-\calE(\bv,\bbE)}\dd t'
\end{aligned}
\end{equation}
for a.a.~$s,t\in[0,T]$ with $s<t$ 
and all $\psi\in C_c^1([0,T)\times\ol\Omega;\R^3)$
and $\Psi\in C_c^1([0,T)\times\ol\Omega;\R_{\sym}^{3\times 3})$
such that $\psi=0$ on $(0,T)\times\Gamma_{\rmD}$ and 
$\psi\cdot\bn=0$ on $(0,T)\times\Gamma_{\rmN}$.
\end{definition}

Compared to the weak formulation~\eqref{eq:weakform.original},
there are two major differences to~\eqref{eq:envarform}.
Firstly, there occurs the new energy variable $E$, which dominates 
the state-dependent energy $\calE(\bv,\bbE)$ and 
replaces it on the left-hand side of the inequality.
Secondly, there is a non-zero term on the right-hand side,
composed of the product of the \textit{energy defect} $E-\calE(\bv,\bbE)\geq 0$
and the \textit{regularity weight} $\calK$.
In other words, the energy-variational formulation~\eqref{eq:envarform}
coincides with the weak formulation~\eqref{eq:weakform.original}
up to an error that is controlled by the energy defect.
In particular, if $E=\calE(\bv,\bbE)$, then
the two formulations coincide,
so that every weak solution is also an energy-variational solution.

Invoking Lemma~\ref{lem:ineq.pointwisevariational},
we obtain directly from the definition that
the state variables $\bv$, $\bbE$ of an energy-variational solution are weakly* continuous with respect to time.
Here, duality is to be understood with respect to
the space of test functions.
Moreover, $(\bv,\bbE)$ attains the initial values in this sense,
compare Remark~\ref{rem:weakform.variational}.

Clearly, the set of energy-variational solutions depends on $\calK$,
and it becomes larger if we increase $\calK$.
It thus seems reasonable find $\calK$
as small as possible such that solutions exist. 
To this end,
we consider the regularity weight 
defined by
\begin{equation}
\label{eq:K.lsc}
\calK(\psi,\Psi)
=\frac{M}{\kappa}\norm{\nabla\psi}_{L^\infty(\Omega)}
+\norm{\dv\psi}_{L^\infty(\Omega)}
+\frac{5}{\kappa\mu}\norm{\Psi}_{L^\infty(\Omega)}^2.
\end{equation}
As we shall see in Lemma~\ref{lem:K.giveslsc} below,
this choice is appropriate to allow for a limit passage in~\eqref{eq:envarform}.
Note that for the construction of approximate solutions 
in Section~\ref{sec:approximation},
we will use a different regularity weight $\wt{\calK}$ 
that plays a similar role and yields suitably convexity properties,
see Lemma~\ref{lem:Kconvexifies} below.

\begin{theorem}
\label{thm:existence.envar}
Let \ref{ass:Omega}--\ref{ass:PDphi} be satisfied.
Let $\bv_0\in L^1(\Omega;\R^3)$ and $\bbE_0\in L^1(\Omega;\R^{3\times3})$  
such that $\calE(\bv_0,\bbE_0)<\infty$.
Then there exists an energy-variational solution $(\bv,\bbE)$ to~\eqref{eq:system.reg}
in the sense of Defintion~\ref{def:envar},
which satisfies $\lim_{t\searrow0}E(t)=\calE(\bv_0,\bbE_0)$,
and the associated regularity weight $\calK$ is given by~\eqref{eq:K.lsc}.
\end{theorem}

The proof of Theorem~\ref{thm:existence.envar}
will be concluded in Subsect.\,\ref{subsec:limit.envar}.

\begin{remark}
As we shall see from the proof, 
we can further refine the solution concept from Defintion~\ref{eq:envarform}.
By the strong compactness in $\bv$,
we actually find $E=\int_\Omega \frac{\rho}{2}|\bv|^2\dd x + I$
with $I\geq \int_\Omega\varphi(\bbE)\dd x$ a.e.~in $(0,T)$.
Therefore, 
the energy defect $E-\calE(\bv,\bbE)=I-\int_\Omega\varphi(\bbE)\dd x\geq 0$ 
is only due to the stored elastic energy.
\end{remark}

\subsection{Functional framework}

We introduce some abbreviations.
For $\gamma\geq 0$, we define
\begin{equation}
\label{eq:DN.def}
\begin{aligned}
\calD_\gamma(\bv,\bbE)
&\coloneqq \int_\Omega \bbS(\bv):\bbD(\bv)
+\gamma\absl{\nabla [D\varphi(\bbE)]}^2 
\dd x,
\\
\calN_\gamma(\bv,\bbE\mid\psi,\Psi)
&\coloneqq\int_\Omega
\frac{\rho}{2}(\dv\bv)(\bv\cdot\psi)
+\rho (\bv\otimes\bv):\nabla\psi \dd x
\\
&\quad
+\int_\Omega
-(D\varphi(\bbE)+\varphi(\bbE)\bbI):\nabla\psi -\bbS(\bv):\bbD(\psi)
\dd x
\\
&\quad
+\int_\Omega
(\dv\bv)\bbE:\Psi+\bv\cdot\nabla \Psi:\bbE - (\bbE\bbW(\bv)-\bbW(\bv)\bbE):\Psi
\dd x 
\\
&\quad
+\int_\Omega
\bbD(\bv):\Psi - \gamma \nabla[D\varphi(\bbE)]\tcolon\nabla\Psi 
 \dd x,
\\
\calB(t;\psi)
&\coloneqq\int_\Omega\bf(t)\cdot\psi\dd x.
\end{aligned}
\end{equation}
Moreover, we set $\calD_0(\bv):=\calD_0(\bv,\bbE)$,
which is independent of $\bbE$.
As we only consider quadratic elastic energies as in~\ref{ass:phi.quadratic} if $\gamma>0$,
the functional $\calD_\gamma$ is actually given by
\[
\calD_\gamma(\bv,\bbE)
\coloneqq \int_\Omega \bbS(\bv):\bbD(\bv)
+\gamma\absl{\nabla \bbK[\bbE]}^2 
\dd x
\]
and thus quadratic and strictly convex.
We further introduce the function spaces
\[
\calX_\gamma=\calX^{\bv}\times\calX_\gamma^{\bbE},
\qquad
\calY=\calY^{\bv}
\times\calY^{\bbE}
\]
with
\[
\begin{aligned}
\calX^{\bv}&=\setcl{\bv\in H^1(\Omega;\R^3)}{\bv=0 \text{ on }\Gamma_{\rmD},\ \bv\cdot\bn=0 \text{ on }\Gamma_{\rmN}=0},
\\
\calX^{\bbE}_\gamma&=
\begin{cases}
L^2(\Omega;\R^{3\times3}_{\sym}) &\text{if }\gamma=0,
\\
H^1(\Omega;\R^{3\times3}_{\sym}) &\text{if }\gamma>0,
\end{cases}
\end{aligned}
\]
and
\[
\begin{aligned}
\calY^{\bv}
&=
\setcl{\psi\in C^1(\overline{\Omega};\R^3)}{\psi=0\text{ on }\Gamma_D,\ \psi\cdot\bn=0\text{ on }\Gamma_N},
\\
\calY^{\bbE}
&=
C^1(\overline{\Omega};\R^{3\times3}_{\sym}).
\end{aligned}
\]

\subsection{Consistency}

From the derivation of the solution concepts,
we have seen that strong solutions to~\eqref{eq:system} and~\eqref{eq:system.reg} 
with sufficient regularity 
are generalized solutions in the sense of Definition~\ref{def:weaksol.gamma} 
and Definition~\ref{def:envar}, respectively.
It is a natural question if also the converse it true, 
that is if generalized solutions with increased regularity
satisfy the problem in the classical sense.  

We first show this consistency property for weak solutions to~\eqref{eq:system.reg}.
To this end, we assume 
\begin{equation}
\begin{aligned}
\label{eq:approximation.P}
&\forall\, \bbA\in L^1(\Omega;\R^{3\times3}_{\sym,0})\ \
\forall\, \bbG\in C^2(\partial\Omega;\R^{3\times3}_{\sym,0})\ \
\exists\, (\Phi_j)_{j\in\N}\subset C^1(\overline{\Omega};\R^{3\times3}_{\sym,0}):
\\
&\qquad
\Phi_j\big|_{\partial\Omega}=\bbG,
\qquad
\Phi_j \tow \bbA \text{ in }L^1(\Omega;\R^{3\times3}_{\sym,0}),
\qquad
\calP(\Phi_j)\to\calP(\bbA) \text{ as }j\to\infty.
\end{aligned}
\end{equation}
One readily shows that this approximation property 
is satisfied for integral functionals of the form 
$\calP(\bbA)=\int_\Omega P(\bbA(x))\dd x$
if $P\colon\R^{3\times3}_{\sym,0}\to[0,\infty]$ is a convex function 
with $P(0)=0$ and with a suitable upper bound.
In particular, the visco-plastic dissipation potential~\eqref{eq:disspot.example} 
is suitable. 

\begin{proposition}
\label{prop:cons.weak.reg}
Let $\calP$ satisfy~\eqref{eq:approximation.P}.
Let $(\bv,\bbE)$ be a weak solution to~\eqref{eq:system.reg} 
in the sense of Definition~\ref{def:weaksol.gamma}.
Assume 
\[
\bv\in C^1([0,T];C^2(\ol\Omega;\R^3)),
\qquad
\bbE\in C^1([0,T];C^2(\ol\Omega;\R^{3\times3}_{\sym})).
\]
Then $(\bv,\bbE)$ is a strong solution, that is, 
the system~\eqref{eq:system.reg} is satisfied pointwise.
\end{proposition}

\begin{proof}
We use the abbreviations introduced in~\eqref{eq:DN.def}.
Due to the regularity assumptions, we can write
\[
\partial_t\calE(\bv,\bbE)=\int_\Omega\rho\partial_t\bv\cdot\bv + \partial_t\bbE:D\varphi(\bbE) \dd x.
\]
Using integration by parts in time, 
we can thus reformulate~\eqref{eq:weakform.reg} as
\[
\begin{aligned}
&\int_s^t \int_\Omega\rho\partial_t\bv\cdot\bv + \partial_t\bbE:D\varphi(\bbE) \dd x
+\calD_\gamma(\bv,\bbE)
+\calP(\dev D\varphi(\bbE))-\calB(\tau;\bv) \dd \tau
\\
&\quad
+\int_s^t\int_\Omega
-\rho\partial_t\bv\cdot\psi
-\partial_t\bbE:\Psi\dd x
+\calN_\gamma(\bv,\bbE\mid\psi,\Psi)
-\calP(\dev\Psi)+\calB(\tau;\psi)\dd \tau
\leq 0
\end{aligned}
\]
for all $s,t\in[0,T]$ with $s<t$,
and
for all test functions 
$(\psi,\Psi)\in C^1([0,T];\calY)$.
In particular, we can omit the time integrals and consider 
time-independent test functions $(\psi,\Psi)\in \calY$.
We set $\Sigma=\bbS(\bv)+D\varphi(\bbE)+\varphi(\bbE)\bbI$.
Employing integration by parts 
and, in some sense, reverting the algebra that led to the weak formulation~\eqref{eq:weakform.original},
we arrive at
\[
\begin{aligned}
&\int_\Omega
\Lp{\rho\partial_t\bv+\rho\bv\cdot\nabla\bv+\frac{\rho}{2}(\dv\bv)\bv
-\dv\Sigma-\bf
}\cdot\lp{\bv-\psi}  \dd x 
\\
&\quad
+\int_\Omega
\Lp{\jder\bbE-\bbD(\bv)-\gamma\Delta[D\varphi(\bbE)]}:\lp{D\varphi(\bbE)-\Psi} \dd x 
\\
&\quad
+\int_{\Gamma_N} \Sigma\bn\cdot\lp{\bv-\psi}\dd \sigma
+\int_{\partial\Omega} \gamma\bn\cdot\nabla[D\varphi(\bbE)]:\lp{D\varphi(\bbE)-\Psi} \dd \sigma
\\
&\leq \calP(\dev\Psi) - \calP(\dev D\varphi(\bbE))
\end{aligned}
\]
in $(0,T)$ for all $(\psi,\Psi)\in \calY$.
On the one hand, the choice $\Psi=D\varphi(\bbE)$ yields
\[
\int_\Omega
\Lp{\rho\partial_t\bv+\rho\bv\cdot\nabla\bv+\frac{\rho}{2}(\dv\bv)\bv
-\dv\Sigma-\bf
}\cdot\lp{\bv-\psi}  \dd x 
+\int_{\Gamma_N} \Sigma\bn\cdot\lp{\bv-\psi}\dd \sigma
\leq 0
\]
for all $\psi\in\calY^{\bv}$.
By the fundamental theorem of calculus, 
this implies the momentum equation~\eqref{eq:sys.mom.reg} 
and the boundary conditions~\eqref{eq:neumannbdry.reg}
as $\bv\cdot\bn=\psi\cdot\bn=0$.
On the other hand, the choice $\psi=\bv$ leads to
\[
\begin{aligned}
&\int_\Omega
\Lp{\jder\bbE-\bbD(\bv)-\gamma\Delta[D\varphi(\bbE)]}:\lp{D\varphi(\bbE)-\Psi} \dd x 
+\int_{\partial\Omega} \gamma\bn\cdot\nabla[D\varphi(\bbE)]:\lp{D\varphi(\bbE)-\Psi} \dd \sigma
\\
&\quad\leq \calP(\dev\Psi) - \calP(\dev D\varphi(\bbE))
\end{aligned}
\]
for all $\Psi\in\calY^{\bbE}=C^1(\overline{\Omega};\R^{3\times3}_{\sym})$.
If we assume $\Psi=D\varphi(\bbE)$ on $\partial\Omega$, 
the boundary integral vanishes,
and by approximation property~\eqref{eq:approximation.P}, 
the resulting inequality extends to all $\Psi\in L^1(\Omega;\R^{3\times3}_{\sym})$,
which gives~\eqref{eq:sys.strain.reg}
by the definition of the convex subdifferential.
A combination with the previous inequality further yields~\eqref{eq:sys.bdry.E.reg}.
The other boundary conditions are satisfied by assumption.
Since the initial values are attained in a weak* sense
for any weak solution, 
compare Remark~\ref{rem:weakform.variational},
the assumed regularity yields~\eqref{eq:initial.reg}
pointwise.
In summary, we have shown that $(\bv,\bbE)$ is a classical solution to~\eqref{eq:system.reg}.
\end{proof}

An analogous statement holds for energy-variational solutions to~\eqref{eq:system}.
Here the approximation property~\eqref{eq:approximation.P} can be relaxed to
\begin{equation}
\begin{aligned}
\label{eq:approximation.P.weaker}
&\forall\, \bbA\in L^1(\Omega;\R^{3\times3}_{\sym,0})\ \
\exists\, (\Phi_j)_{j\in\N}\subset C^1(\overline{\Omega};\R^{3\times3}_{\sym,0}):
\\
&\qquad
\Phi_j \tow \bbA \text{ in }L^1(\Omega;\R^{3\times3}_{\sym,0}),
\qquad
\calP(\Phi_j)\to\calP(\bbA) \text{ as }j\to\infty.
\end{aligned}
\end{equation}
Observe that the consistency result is independent of the regularity weight $\calK$.

\begin{proposition}
\label{prop:cons.envar}
Let $\calP$ satisfy~\eqref{eq:approximation.P.weaker}.
Let $(\bv,\bbE,E)$ be an energy-variational solution to~\eqref{eq:system} 
in the sense of Definition~\ref{def:envar}.
Assume $\lim_{t\searrow0}E(t)=\calE(\bv_0,\bbE_0)$ and
\[
\bv\in C^1([0,T];C^2(\ol\Omega;\R^3)),
\qquad
\bbE\in C^1([0,T];C^1(\ol\Omega;\R^{3\times3}_{\sym})).
\]
Then $E=\calE(\bv,\bbE)$, and $(\bv,\bbE)$ is a strong solution, that is, 
the system~\eqref{eq:system} is satisfied pointwise.
\end{proposition}

\begin{proof}
In virtue of the regularity assumptions, we can use $(\psi,\Psi)=(\bv,D\varphi(\bbE))$ as 
a test function in the energy-variational formulation~\eqref{eq:envarform}.
A straightforward calculation leads to
\[
\Lb{E-\calE(\bv,\bbE)}\Big|_s^t
\leq \int_s^t \mathcal K(\bv,D\varphi(\bbE))\lp{E-\calE(\bv,\bbE)}\dd t'
\]
for all a.a.~$s<t$.
By Gronwall's inequality, this implies $E=\calE(\bv,\bbE)$ a.e.~in $[0,T]$.
In particular, $(\bv,\bbE)$ is a weak solution and subject to~\eqref{eq:weakform.original}.
Proceeding now as in the proof of Proposition~\ref{prop:cons.weak.reg},
we conclude that $(\bv,\bbE)$ is even a strong solution to~\eqref{eq:system}.
Here, it is sufficient to assume~\eqref{eq:approximation.P.weaker} instead of~\eqref{eq:approximation.P}
since no boundary integral over $\partial\Omega$ appears in this argument 
as $\gamma=0$.
\end{proof}

\section{Time-discrete approximation}
\label{sec:approximation}

For both existence results, we use the same approximation scheme,
based on the weak formulations~\eqref{eq:weakform.original} and~\eqref{eq:weakform.reg}.
With the abbreviations from~\eqref{eq:DN.def},
those can be summarized as
\begin{equation}\label{eq:weakform.combined}
\begin{aligned}
&\Big[\calE(\bv,\bbE)
-\int_\Omega\rho\bv\cdot\psi + \bbE:\Psi \dd x\Big]\Big|_s^t
+\int_s^t \calD_\gamma(\bv,\bbE)
+\calP(\dev D\varphi(\bbE))-\calB(\tau;\bv) \dd \tau
\\
&\quad
+\int_s^t\int_\Omega
\rho\bv\cdot\partial_t\psi
+\bbE:\partial_t\Psi\dd x
+\calN_\gamma(\bv,\bbE\mid\psi,\Psi)
-\calP(\dev\Psi)+\calB(\tau;\psi)\dd \tau
\leq 0
\end{aligned}
\end{equation}
for a.a.~$s,t\in[0,T]$ with $s<t$,
and
for all test functions 
$(\psi,\Psi)\in C^1([0,T];\calY)$.

\subsection{Time-iterative scheme}

To introduce the time discretization at level $N\in\N$,
we fix the time-step size $\tau=T/N$,
and the intermediate time steps 
$t^n:=n\tau$ with $n=0,\dots,N$.
We define the time-discrete version $\calB^n$ of the functional $\calB$,
associated with the external force, by
\begin{equation}
\label{eq:Bnfn}
\calB^n(\psi)=\int_\Omega \bf^n\cdot\psi\dd x, 
\qquad
\bf^n(x)=\frac{1}{\tau}\int_{t^{n-1}}^{t^n} \bf(t,x)\dd t.
\end{equation}
We let $(\bv^0,\bbE^0)=(\bv_0,\bbE_0)$, 
and we determine $(\bv^n,\bbE^n)\in\calX_\gamma$, for $n=1,\dots,N$,
iteratively by 
\begin{equation}
\label{eq:weakform.discretized}
\calF^n_\tau(\bv^n,\bbE^n\mid\psi,\Psi)\leq 0,
\end{equation}
for all test functions
$(\psi,\Psi)\in C^1([0,T];\calY)$
with $\tau\wt{\calK}(\psi,\Psi)\leq 1$.
Here, 
we have set
\begin{equation}
\label{eq:F.def}
\begin{aligned}
&\calF^n_\tau(\bv,\bbE\mid\psi,\Psi)
\\
& \quad
:=\calE(\bv,\bbE)-\calE(\bv^{n-1},\bbE^{n-1})
-\int_\Omega\rho(\bv-\bv^{n-1})\cdot\psi\dd x
-\int_\Omega(\bbE-\bbE^{n-1}):\Psi\dd x
\\
&\qquad
+\tau \Lp{\calD_\gamma(\bv,\bbE)
+\calP(\dev D\varphi(\bbE))
+\calN_\gamma(\bv,\bbE\mid\psi,\Psi)
-\calP(\dev\Psi)-\calB^n(\bv-\psi)},
\end{aligned}
\end{equation}
and the function $\wt{\calK}\colon\calY\to[0,\infty)$ is defined by
\begin{equation}
\label{eq:K.convexify}
\begin{aligned}
\wt{\calK}(\psi,\Psi)
&= \frac{2\rho}{\mu} \|\psi\|_{L^\infty(\Omega)}^2
+\max\setL{2,\frac{M}{\kappa}} \|\nabla\psi\|_{L^\infty(\Omega)}
\\
&\qquad
+ \|\dv\psi\|_{L^\infty(\Omega)}
+ \frac{6}{\kappa\mu} \|\Psi\|_{L^\infty(\Omega)}^2
+ \frac{1}{\kappa} \|\nabla\Psi\|_{L^\infty(\Omega)}.
\end{aligned}
\end{equation}
Note that this function $\wt{\calK}$ differs from the 
regularity weight $\calK$ from~\eqref{eq:K.lsc},
which allows for the limit passage towards 
an energy-variational solution.
The possibility to use a different regularity weight for the construction of approximate solutions 
was recently discovered and exploited in~\cite{EiterLasarzikSliwinski_exselenvar}.

By choosing $(\psi,\Psi)=(0,0)$ in~\eqref{eq:weakform.discretized},
we obtain the time-discrete energy-dissipation inequality
\begin{equation}
\label{eq:enineq.discrete}
\calE(\bv^n,\bbE^n)-\calE(\bv^{n-1},\bbE^{n-1})
+\tau \calD_\gamma(\bv^n,\bbE^n)
+\tau \calP(\dev D\varphi(\bbE^n))
-\calB^n(\bv)\leq 0.
\end{equation}

Observe that~\eqref{eq:weakform.discretized} 
corresponds to a fully implicit scheme for~\eqref{eq:weakform.combined},
where 
we have restricted the class of admissible test functions
by means of the function $\wt{\calK}$.
This restriction ensures the existence of the iterates in terms 
of a minimization problem
due to the convexity of $\calF^n_\tau$ with respect to the state variables.

\begin{lemma}
\label{lem:Kconvexifies}
Let $\gamma\geq 0$, and let~\ref{ass:Omega}--\ref{ass:PDphi} be satisfied.
If $\gamma>0$, we also assume~\ref{ass:phi.quadratic}.
Let $n\in\set{1,\dots,N}$, $N\in\N$, and let $(\bv^{n-1},\bbE^{n-1})\in\calX_\gamma$.
Let $\wt{\calK}\colon\calY\to[0,\infty)$ be given by~\eqref{eq:K.convexify}.
For every $(\psi,\Psi)\in\calY$
with $\tau\wt{\calK}(\psi,\Psi)\leq 1$,
the function
\[
\calX_\gamma\to\R,
\quad
(\bv,\bbE)\mapsto\calF^n_\tau(\bv,\bbE\mid\psi,\Psi)
\]
is convex and weakly lower semicontinuous
and satisfies the estimate
\begin{equation}
\label{eq:F.est}
\begin{aligned}
\calF^n_\tau&(\bv,\bbE\mid\psi,\Psi)
\geq 
\lb{1-\tau \wt{\calK}(\psi,\Psi)}\calE(\bv,\bbE)-\calE(\bv^{n-1},\bbE^{n-1})
\\
&\quad-\int_\Omega\rho(\bv-\bv^{n-1})\cdot\psi\dd x
-\int_\Omega(\bbE-\bbE^{n-1}):\Psi\dd x
\\
&\quad
+\tau \llp{\int_\Omega\gamma\absl{\nabla\nb{D\varphi(\bbE)}}^2
-\bf^n\cdot\bv\dd x
+\calP(\dev D\varphi(\bbE))
-\calP(\dev\Psi)}
\\
&\quad
+\tau\int_\Omega
-\bbS(\bv):\bbD(\psi)+\bf^n\cdot\psi
-\bbD(\bv):\Psi - \gamma \nabla[D\varphi(\bbE)]\tcolon\nabla\Psi 
\dd x.
\end{aligned}
\end{equation}
\end{lemma}

\begin{proof}
Let $(\psi,\Psi)\in\calY$
with $\tau\wt{\calK}(\psi,\Psi)\leq 1$.
We add and subtract $\tau \wt{\calK}(\psi,\Psi)\calE(\bv,\bbE)$
and rearrange the terms in $\calF^n_\tau$  to obtain
\begin{equation}\label{eq:F.rearranged}
\begin{aligned}
\calF^n_\tau&(\bv,\bbE\mid\psi,\Psi)
=\lb{1-\tau \wt{\calK}(\psi,\Psi)}\calE(\bv,\bbE)-\calE(\bv^{n-1},\bbE^{n-1})
\\
&\qquad-\int_\Omega\rho(\bv-\bv^{n-1})\cdot\psi\dd x
-\int_\Omega(\bbE-\bbE^{n-1}):\Psi\dd x
\\
&\qquad
+\tau \llp{\int_\Omega\gamma\absl{\nabla\nb{D\varphi(\bbE)}}^2
-\bf^n\cdot\bv\dd x
+\calP(\dev D\varphi(\bbE))
-\calP(\dev\Psi)}
\\
&\qquad
+\tau\Lp{\calN_\gamma(\bv,\bbE\mid\psi,\Psi)
+\calD_0(\bv)
+\wt{\calK}(\psi,\Psi)\calE(\bv,\bbE)}.
\end{aligned}
\end{equation}
In virtue of the restriction $\tau\wt{\calK}(\psi,\Psi)\leq 1$ and the assumptions \ref{ass:Omega}--\ref{ass:PDphi},
and~\ref{ass:phi.quadratic} if $\gamma>0$,
the functional defined by the first three lines 
satisfies the assertions on convexity and lower semicontinuity.
It thus remains to show these properties for the functional defined by the last line,
that is, for
\[
\calG\colon\calX_\gamma\to\R,
\quad
\calG(\bv,\bbE)=\calN_\gamma(\bv,\bbE\mid\psi,\Psi)
+\calD_0(\bv)
+\wt{\calK}(\psi,\Psi)\calE(\bv,\bbE),
\]
which we split into several parts.
We first use Young's inequality to estimate
\[
\begin{aligned}
\calG_1(\bv,\bbE)
&:=
\int_\Omega
\frac{\rho}{2}(\dv\bv)(\bv\cdot\psi) + \frac{\mu}{4}|\nabla\bv|^2 \dd x + \frac{2\rho}{\mu}\norm{\psi}_{L^\infty(\Omega)}^2\int_\Omega\frac{\rho}{2}|\bv|^2\dd x
\geq
0,
\\
\calG_2(\bv,\bbE)
&:=
\int_\Omega\rho (\bv\otimes\bv):\nabla\psi \dd x+2\norm{\nabla\psi}_{L^\infty(\Omega)}\int_\Omega\frac{\rho}{2}|\bv|^2\dd x
\geq0,
\\
\calG_3(\bv,\bbE)
&:=
\int_\Omega
(\dv\bv)\bbE:\Psi
+\frac{\mu}{4}|\nabla\bv|^2 \dd x 
+\frac{2}{\kappa\mu}\norm{\Psi}_{L^\infty(\Omega)}^2\int_\Omega\frac{\kappa}{2}|\bbE|^2\dd x
\geq 0,
\\
\calG_4(\bv,\bbE)
&:=
\int_\Omega
\bv\cdot\nabla \Psi:\bbE
+\frac{1}{\rho}\norm{\nabla\Psi}_{L^\infty(\Omega)}
\int_\Omega\frac{\rho}{2}|\bv|^2\dd x
+\frac{1}{\kappa}\norm{\nabla\Psi}_{L^\infty(\Omega)}
\int_\Omega\frac{\kappa}{2}|\bbE|^2\dd x
\geq 0,
\\
\calG_5(\bv,\bbE)
&:=
\int_\Omega
-(\bbE\bbW(\bv)-\bbW(\bv)\bbE):\Psi
+\frac{\mu}{2}|\nabla\bv|^2 \dd x 
+\frac{4}{\kappa\mu}\norm{\Psi}_{L^\infty(\Omega)}^2\int_\Omega\frac{\kappa}{2}|\bbE|^2\dd x
\geq 0.
\end{aligned}
\]
Hence, $\calG_j$, $j=1,\dots,5$, is a non-negative and 
quadratic functional.
Therefore, it is convex,
and as it is defined in terms of a convex integrand, 
it is also weakly lower semicontinuous.
The same applies to the functionals defined by
\[
\begin{aligned}
\calG_6(\bv,\bbE)
&:=\int_\Omega-D\varphi(\bbE):\nabla\psi \dd x
+ \frac{M}{\kappa}\norm{\nabla\psi}_{L^\infty}\int_\Omega\frac{\kappa}{2}|\bbE|^2 \dd x,\\
\calG_7(\bv,\bbE)
&:=\int_\Omega-(\varphi(\bbE)\bbI):\nabla\psi \dd x
+\norm{\dv\psi}_{L^\infty}\int_\Omega\varphi(\bbE)\dd x
\end{aligned}
\]
due to $|D^3\varphi(\bbE)|\leq M$ and the convexity of $\varphi$, respectively.
Similarly, the Korn-type inequality~\eqref{eq:Korn} implies
that
\[
\calH(\bv):=\calD_0(\bv)-\int_\Omega\mu|\nabla\bv|^2\dd x
\]
defines a convex and weakly lower semicontinuous functional.
We now write $\calG$ as
\[
\begin{aligned}
\calG(\bv,\bbE)
&=\sum_{j=1}^7\calG_j(\bv,\bbE) 
+\calH(\bv)
+\int_\Omega
-\bbS(\bv){:}\bbD(\psi)
+\bf{\cdot}\psi
-\bbD(\bv){:}\Psi 
- \gamma \nabla[D\varphi(\bbE)]{\tcolon}\nabla\Psi
\dd x
\\
&\quad
+\wt{\calK}(\psi,\Psi)\calE(\bv,\bbE)- \Lp{
\frac{2\rho}{\mu} \|\psi\|_{\infty}^2
+2 \|\nabla\psi\|_{\infty}}
\int_\Omega\frac{\rho}{2}|\bv|^2\dd x
\\
&\quad
- \|\dv\psi\|_\infty \int_\Omega\varphi(\bbE)\dd x
-\Lp{\frac{M}{\kappa}\|\nabla\psi\|_{\infty}
+ \frac{6}{\kappa\mu} \|\Psi\|_\infty
+ \frac{1}{\kappa} \|\nabla\Psi\|_\infty
}\int_\Omega\frac{\kappa}{2}|\bbE|^2\dd x.
\end{aligned}
\]
From the strong convexity of $\varphi$
and the choice of $\wt{\calK}$ as in~\eqref{eq:K.convexify},
and from~\ref{ass:phi.quadratic} if $\gamma>0$,
we see that $\calG$ is convex and weakly lower semicontinuous.
Due to the non-negativity of $\calG_1,\dots,\calG_7$, $\calH$, 
we further conclude
\[
\calG(\bv,\bbE)\geq \int_\Omega
-\bbS(\bv):\bbD(\psi)+\bf^n\cdot\psi
-\bbD(\bv):\Psi - \gamma \nabla[D\varphi(\bbE)]\tcolon\nabla\Psi
\dd x,
\]
whence we obtain~\eqref{eq:F.est} from~\eqref{eq:F.rearranged}.
\end{proof}

To show existence of a suitable pair $(\bv^n,\bbE^n)$ 
satisfying~\eqref{eq:weakform.discretized} at the $n$-th time step, 
we determine $(\bv^n,\bbE^n)$ by solving the following saddle-point problem.

\begin{lemma}
\label{lem:saddlepoint.existence}
Let $\gamma\geq 0$, and let~\ref{ass:Omega}--\ref{ass:PDphi} be satisfied.
If $\gamma>0$, we also assume~\ref{ass:phi.quadratic}.
Let $n\in\set{1,\dots,N}$, $N\in\N$, and let $(\bv^{n-1},\bbE^{n-1})\in\calX_\gamma$ be given.
Then there exists $(\bv^{n},\bbE^{n})\in\calX_\gamma$ such that
\begin{equation}
\label{eq:saddlepoint}
(\bv^n,\bbE^n)
\in\argmin_{(\bv,\bbE)\in \calX_\gamma}\sup_{\tau\wt{\calK}(\psi,\Psi)\leq 1}
\calF^n_\tau(\bv,\bbE\mid\psi,\Psi).
\end{equation}
Moreover, $(\bv^{n},\bbE^{n})$ satisfies~\eqref{eq:weakform.discretized}.
\end{lemma}

\begin{proof}
From the definition in~\eqref{eq:F.def}, 
we see immediately that 
for fixed $(\bv,\bbE)\in\calX_\gamma$,
the functional
$(\psi,\Psi)\mapsto\calF^n_\tau(\bv,\bbE\mid\psi,\Psi)$
is concave and upper semicontinuous
on the convex set defined by $\tau\wt{\calK}(\psi,\Psi)\leq 1$.
By Lemma~\ref{lem:Kconvexifies},
the functional $(\bv,\bbE)\mapsto\calF^n_\tau(\bv,\bbE\mid\psi,\Psi)$
is convex and weakly lower semicontinuous
for fixed $(\psi,\Psi)\in\calY$ with $\tau\wt{\calK}(\psi,\Psi)\leq 1$.
Moreover, $(\bv,\bbE)\mapsto\calF^n_\tau(\bv,\bbE\mid 0,0)$ has weakly compact sublevel sets due to the coercivity estimate
\[
\begin{aligned}
\calF^n_\tau(\bv,\bbE\mid 0,0)
&\geq \calE(\bv,\bbE)-\calE(\bv^{n-1},\bbE^{n-1}) + \tau\calD_\gamma(\bv,\bbE)
-\tau\calB^n(\bv)
\\
&\geq \int_\Omega\frac{\rho}{2}\abs{\bv}^2+\beta|\bbE|^2-1\dd x
-\tau\norm{\bf^{n}}_{L^2(\Omega)}\norm{\bv}_{L^2(\Omega)},
\end{aligned}
\]
where we employed~\eqref{eq:phi.quadrgrowth} for suitable $\beta>0$.
From a classical saddle-point theorem, 
see~\cite[Thm.\,2.130, Rem.\,2.131]{BarbuPrecupanu2012convopti} for instance,
we thus obtain the existence of 
$(\bv^n,\bbE^n)\in\calX_\gamma$ such that
\[
\begin{aligned}
\sup_{\tau\wt{\calK}(\psi,\Psi)\leq 1}\calF^n_\tau(\bv^n,\bbE^n\mid\psi,\Psi)
&=\min_{(\bv,\bE)\in \calX_\gamma}\sup_{\tau\wt{\calK}(\psi,\Psi)\leq 1}
\calF^n_\tau(\bv,\bbE\mid\psi,\Psi)
\\
&=\sup_{\tau\wt{\calK}(\psi,\Psi)\leq 1}\min_{(\bv,\bE)\in \calX_\gamma}
\calF^n_\tau(\bv,\bbE\mid\psi,\Psi).
\end{aligned}
\]
In particular, $(\bv^n,\bbE^n)$ satisfies~\eqref{eq:saddlepoint},
and we have
\[
\sup_{\tau\wt{\calK}(\psi,\Psi)\leq 1}\calF^n_\tau(\bv^n,\bbE^n\mid\psi,\Psi)
\leq\sup_{\tau\wt{\calK}(\psi,\Psi)\leq 1}\calF^n_\tau(\psi,D\varphi^*(\Psi)\mid\psi,\Psi).
\]
Here, $\varphi^*$ is the Fenchel conjugate of $\varphi$,
defined by
$
\varphi^*(\bbB)=\sup\setc{\bbB:\bbA-\varphi(\bbA)}{\bbA\in\R^{3\times3}_{\sym}}.
$
It holds $[D\varphi]^{-1}(\bbB)=D\varphi^*(\bbB)$, and
$\varphi(D\varphi^*(\bbB))+\varphi^*(\bbB)=D\varphi^*(\bbB):\bbB$.
With these properties, integration by parts and~\ref{ass:phi}\ref{ass:phi.commute},
we obtain for any $(\psi,\Psi)$ that
\[
\begin{aligned}
&\calF^n_\tau(\psi,D\varphi^*(\Psi)\mid\psi,\Psi)
\\
&=
\calE(\psi,D\varphi^*(\Psi))-\calE(\bv^{n-1},\bbE^{n-1})
-\int_\Omega\rho(\psi-\bv^{n-1})\cdot\psi\dd x
-\int_\Omega(D\varphi^*(\Psi)-\bbE^{n-1}):\Psi\dd x
\\
&\qquad
+\tau\int_\Omega
\frac{\rho}{2}(\dv\psi)|\psi|^2
+\rho (\psi\otimes\psi):\nabla\psi \dd x
-\Psi:\nabla\psi -\varphi(D\varphi^*(\Psi))\dv\psi 
 \dd x
\\
&\qquad
+\tau\int_\Omega
(\dv\psi)D\varphi^*(\Psi):\Psi+\psi\cdot\nabla \Psi:D\varphi^*(\Psi)
\dd x
\\
&\qquad
+\tau\int_\Omega
 - (D\varphi^*(\Psi)\bbW(\psi)-\bbW(\psi)D\varphi^*(\Psi)):\Psi+\bbD(\psi):\Psi 
 \dd x 
\\
&=
\int_\Omega 
-\frac{\rho}{2}|\psi-\bv^{n-1}|^2
-\Lp{\varphi(\bbE^{n-1})-\varphi(D\varphi^*(\Psi))
-D\varphi(D\varphi^*(\Psi)):\lp{\bbE^{n-1}-D\varphi^*(\Psi)}}
\dd x
\\
&\qquad+\tau\int_\Omega 
(\dv\psi)\varphi^*(\Psi)+\psi\cdot\nabla \varphi^*(\Psi)
\dd x
\leq 0,
\end{aligned}
\]
where we used the convexity of $\varphi$ in the last step.
This shows~\eqref{eq:weakform.discretized}
and completes the proof.
\end{proof}

\begin{remark}
The solution $(\bv^n,\bbE^n)$ to the minimization problem~\eqref{eq:saddlepoint}
is unique.
This follows from the uniform convexity 
of the functional $(\bv,\bbE)\mapsto\calF(\bv,\bbE\mid\psi,\Psi)$,
which is inherited from the energy functional $\calE$.
\end{remark}

\begin{remark}\label{rem:phi.quadratic.adaption}
The ensure the inequality~\eqref{eq:Bnfn},
we make use of the saddle-point structure of problem~\eqref{eq:saddlepoint},
To this end, we chose the regularity weight $\wt{\calK}$
in such a way that gives sufficient convexity properties.
This approach, based on an implicit time discretization,
is only possible for $\gamma>0$ if $\varphi$ is quadratic.
A first idea to circumvent this assumption 
would be to discretize the stress diffusion term differently, 
for instance, as
\[
\int_\Omega \gamma\absl{\nabla[D\varphi(\bbE(t))]}^2\dd x
\approx\int_\Omega \sum_{j=1}^3\gamma\absl{D^2\varphi(\bbE^{n-1}):\partial_j\bbE^n}^2\dd x
\]
for $t\in(t^{n-1},t^n)$.
This leads to a convex term in $\bbE^n$,
and the convexity property from Lemma~\ref{lem:Kconvexifies}
and the existence of a minimizer as in~\eqref{eq:saddlepoint}
would still follow for the accordingly modified $\calF^n_\tau$.
However, due to the partially explicit discretization,
the verification of the inequality~\eqref{eq:weakform.discretized}
remains unclear in this case.
\end{remark}

\subsection{Approximate solutions}

We now define the sequence of approximate solution.
Fix $N\in\N$ and let $\tau=T/N$ as before.
We set $(\bv^0,\bbE^0)=(\bv_0,\bbE_0)$ and
let $(\bv^n,\bbE^n)\in\calX_\gamma$ satisfy~\eqref{eq:weakform.discretized}
for $n=1,\dots,N$,
the existence of which we have shown in Lemma~\ref{lem:saddlepoint.existence}.
We introduce the piecewise constant prolongations
$(\bv_N,\bbE_N)\colon (-\tau,T]\to\calX_\gamma$ as
\begin{equation}
\label{eq:approxsol.def}
(\bv_N,\bbE_N)(t)=
(\bv^n,\bbE^n)
\qquad\text{for } t\in(t^{n-1},t^n], \,n=0,\dots,N,
\end{equation}
where $t^n:=n\tau$.

We define piecewise constant approximations of the operator $\calB$ and the 
external force $\bf$ in the same way,
namely
\[
\calB_N(t;\psi)=\int_\Omega \bf_N(t)\cdot\psi\dd x, 
\qquad
\bf_N(t)=\bf^n
\ \text{ for } t\in(t^{n-1},t^n], \,n=1,\dots,N,
\]
for $\bf^n$ as in~\eqref{eq:Bnfn}.
Observe that
\begin{equation}\label{eq:fN.integral}
\int_0^T \norm{\bf_N(t)}_{L^2(\Omega)} \dd t
=\sum_{n=1}^N\tau\norm{\bf^n}_{L^2(\Omega)}
\leq\int_0^T\norm{\bf(t)}_{L^2(\Omega)} \dd t,
\end{equation}
and we have $\bf_N\to\bf$ strongly in $L^1(0,T;L^2(\Omega))$.

For a test function $\phi\in C_c^1([0,T))$,
we define the piecewise constant and piecewise linear approximations
$\phi_N$ and $\wh\phi_N$ as
\[
\begin{aligned}
\phi_N(t)
&=\phi(t^n)
&&\text{for } t\in[t^n,t^{n+1}), \,n=0,\dots,N,
\\
\wh\phi_N(t)
&= \frac{t-t^{n}}{t^{n+1}-t^n}\phi(t^{n+1})
+\frac{t^{n+1}-t}{t^{n+1}-t^{n}}\phi(t^{n})
&&\text{for } t\in[t^{n},t^{n+1}], \,n=0,\dots,N-1.
\end{aligned}
\]
Observe that $\wh\phi_N$ is defined on $[0,T]$,
while $\phi_N$ is defined on $[0,T+\tau)$ with $\phi_N(t)=\phi(T)=0$ for $t\in[T,T+\tau)$.
We further have
the following variant of discrete integration by parts:
For all $a\in L^1(-\tau,T)$ and $\chi\in L^\infty(0,T)$,
it holds
\begin{equation}
\label{eq:intbyparts.disc}
\begin{aligned}
&\int_0^T\phi_N(t)\chi(t)\frac{a(t)-a(t{-}\tau)}{\tau}\dd t
\\
&
=-\Lp{\phi(0) \,\frac{1}{\tau}\int_0^\tau \chi(t)a(t-\tau)\dd t
+\int_\tau^T\partial_t\wh\phi_N(t)\chi(t{+}\tau)a(t)
+\phi_N(t)\delta_\tau\chi(t)a(t)\dd t},
\end{aligned}
\end{equation}
where $\delta_\tau\chi(t):=\frac{1}{\tau}\np{\chi(t+\tau)-\chi(t)}$ denotes the difference quotient.

Now let $\phi\in C_c^1([0,T))$ with $\phi\geq 0$.
Let $(\psi,\Psi)\in C^1([0,T];\calY)$,
and fix $N\in\N$.
Then~\eqref{eq:weakform.discretized} implies
\[
\phi_N(t)\calF^n_\tau(\bv_N(t),\bbE_N(t)\mid\psi(t),\Psi(t))\leq 0
\]
for $t\in(t^{n-1},t^n]$ and $n=1,\dots,N$.
Division by $\tau$, integration over $(0,T)$, and use of identity~\eqref{eq:intbyparts.disc}
lead to
\begin{equation}
\label{eq:discrete.integrated}
\begin{aligned}
&\int_\tau^T-\partial_t\wh\phi_N(t)\Lb{
\calE(\bv_N(t),\bbE_N(t))
-\int_\Omega\rho\bv_N(t)\cdot\psi(t+\tau) + \bbE_N(t):\Psi(t+\tau)
}\dd t
\\
&
+\int_\tau^T\phi_N\Lp{\int_\Omega\rho\bv_N\cdot\delta_\tau\psi
+\bbE_N:\delta_\tau\Psi\dd x}\dd t
-\int_0^T\phi_N\calB_N(\bv_N-\psi)\dd t
\\
&
+\int_0^T\phi_N\Lp{\calD_\gamma(\bv_N,\bbE_N)
+\calP(\dev D\varphi(\bbE_N))
+\calN_\gamma(\bv_N,\bbE_N\mid\psi,\Psi)
-\calP(\dev\Psi)
}\dd t
\\
&
-\phi(0)\Lb{\calE(\bv_0,\bbE_0)
-\int_\Omega\bv_0\cdot\Lp{\frac{1}{\tau}\int_0^\tau\psi(t)\dd t}
+\bbE_0:\Lp{\frac{1}{\tau}\int_0^\tau\Psi(t)\dd t} \dd x}
\leq 0
\end{aligned}
\end{equation}
for all $(\psi,\Psi)\in C^1([0,T];\calY)$
and $\phi\in C_c^1([0,T))$ with $\phi\geq 0$.
We shall pass to the limit $N\to\infty$, that is $\tau\to 0$, in this equation.
To do so, we need convergence in suitable topologies, 
which follows from associated bounds we derive next.

\subsection{A priori estimates}

We now collect the uniform bounds for the approximate solutions $(\bv_N,\bbE_N)$, $N\in\N$,
defined in~\eqref{eq:approxsol.def}.

\begin{lemma}\label{lem:aprioriest}
Let $\gamma\geq 0$, and
let~\ref{ass:Omega}--\ref{ass:PDphi} be satisfied.
If $\gamma>0$, we also assume~\ref{ass:phi.quadratic}.
Then there exists a constant $C>0$
such that 
\begin{align}
\norm{\bv_N}_{L^\infty(0,T;L^2(\Omega))}
+\norm{\nabla\bv_N}_{L^2(0,T;L^2(\Omega))}
&\leq C,
\label{eq:aprioriest.v}
\\
\norm{\bbE_N}_{L^\infty(0,T;L^2(\Omega))}
+\sqrt{\gamma}\norm{\nabla\bbE_N}_{L^2(0,T;L^2(\Omega))}
&\leq C,
\label{eq:aprioriest.E}
\\
\norm{\varphi(\bbE_N)}_{L^\infty(0,T;L^1(\Omega))}
+\norm{D\varphi(\bbE_N)}_{L^\infty(0,T;L^1(\Omega)}
&\leq C,
\label{eq:aprioriest.phi}
\\
\norm{\calE(\bv_N,\bbE_N)}_{\BV([0,T])} 
+\norm{\calP(\dev D\varphi(\bbE_N))}_{L^1(0,T)}
&\leq C
\label{eq:aprioriest.energy}
\end{align}
for all $N\in\N$,
and
\begin{equation}
\norm{\bv_N-\bv_N(\cdot-\tau)}_{L^1(\tau,T;(\calY^{\bv})^*)}
\leq C\tau
\label{eq:aprioriest.increment.v}
\end{equation}
for $N\in\N$ sufficiently large. 
Here $C>0$ is independent of $N=T/\tau$ and $\gamma\geq 0$.
\end{lemma}

\begin{proof}
We fix $N\in\N$ and let $k,\ell\in\N_0$ with
$0\leq k< \ell \leq N$.
Summing up of the discrete energy inequality~\eqref{eq:enineq.discrete}
for $n=k+1,\ldots,\ell$,
we obtain
\begin{equation}
\begin{aligned}
\label{eq:energyest.approx}
&\calE(\bv_N,\bbE_N)\Big|_{t^k}^{t^\ell} 
+\int_{t^k}^{t^\ell}
\calD_\gamma(\bv_N,\bbE_N)
+\calP(\dev D\varphi(\bbE_N))\dd t'
\leq \int_{t^k}^{t^\ell} \calB_N(\bv_N)\dd t'
\\
&\qquad\quad
\leq \int_{t^k}^{t^\ell}\int_\Omega \frac{\rho}{2}\abs{\bv_N}^2 + \frac{1}{2\rho}\abs{\bf}^2\dd x\dd t'
\leq \int_{t^k}^{t^\ell} \calE(\bv_N,\bbE_N) + \frac{1}{2\rho}\norm{\bf}_{L^2(\Omega)}^2 \dd t'
\end{aligned}
\end{equation}
by Young's inequality. 
Using Gronwall's lemma, 
we see that 
the left-hand side of~\eqref{eq:energyest.approx}
is uniformly bounded.
Together with the growth assumption from~\ref{ass:phi}
and the Korn-type inequality~\eqref{eq:Korn},
this immediately yields the uniform bounds~\eqref{eq:aprioriest.v},~\eqref{eq:aprioriest.E}, \eqref{eq:aprioriest.phi}
and~\eqref{eq:aprioriest.energy},
except for the bound on $D\varphi(\bbE_N)$ and $\calE(\bv_N,\bbE_N)$.

The bound on $D\varphi(\bbE_N)$ follows
from assumption~\ref{ass:phi}\ref{ass:phi.D3phibounded}
and Taylor's theorem.
Moreover, we conclude from~\eqref{eq:energyest.approx}
that the function $H_N\colon(0,T)\to\R$ with
\[
H_N(t)=\calE(\bv_N(t),\bbE_N(t))
+\int_{0}^t\calD_\gamma(\bv_N,\bbE_N)
+\calP(\dev D\varphi(\bbE_N))-\calB_N(\bv_N)\dd t'
\]
is nonincreasing,
so that $H_N\in \BV([0,T])$ with
\[
\norm{H_N}_{\BV([0,T])}
=\norm{H_N}_{L^1(0,T)}
+\abs{H_N}_{\TV([0,T])}
\leq (1+T)\calE(\bv_0,\bbE_0) + |H_N(T)| \leq C,
\]
due to the uniform bounds we already established.
With a similar argument, we obtain a uniform bound for 
$H_N-\calE(\bv_N,\bbE_N)$ in $\BV([0,T])$.
In this way, we arrive at a uniform bound
$\norm{\calE(\bv_N,\bbE_N)}_{\BV([0,T])}\leq C$.

It remains to verify~\eqref{eq:aprioriest.increment.v}.
By definition of $\bv_N$, we have
\begin{equation}
\label{eq:incrementest.1}
\norm{\bv_N-\bv_N(\cdot-\tau)}_{L^1(\tau,T;(\calY^{\bv})^*)}
=\sum_{n=1}^{N}\tau\norm{\bv^{n}-\bv^{n-1}}_{(\calY^{\bv})^*}.
\end{equation}
To obtain a bound for the right-hand side, 
we consider $\psi\in \calY^{\bv}$ with $\norm{\psi}_{\calY^{\bv}}\leq 1$.
If we take $N\in\N$ large enough, that is,
$\tau>0$ small enough, 
we have $\tau\wt{\calK}(-\frac{1}{\rho}\psi,0)\leq 1$.
Then we can use~\eqref{eq:weakform.discretized}
and~\eqref{eq:F.est}
with $\Psi=0$
and estimate
\[
\begin{aligned}
&\int_\Omega(\bv^n-\bv^{n-1})\cdot\psi\dd x
\\
&
\leq 
-\Lp{1-\tau \wt{\calK}(-\frac{1}{\rho}\psi,0)}\calE(\bv^n,\bbE^n)+\calE(\bv^{n-1},\bbE^{n-1})
\\
&\quad
-\tau \llp{\int_\Omega\gamma|\nabla[D\varphi(\bbE^n)]|^2
-\bf^n\cdot\bv^n
-\frac{1}{\rho}\bbS(\bv^n)\mathbin{:}\bbD(\psi)
+\frac{1}{\rho}\bf^n\cdot\psi
\dd x
+\calP(\dev D\varphi(\bbE^n))}
\\
&\leq \calE(\bv^{n-1},\bbE^{n-1})-\calE(\bv^n,\bbE^n)
+\tau \wt{\calK}(-\frac{1}{\rho}\psi,0)\calE(\bv^n,\bbE^n)
+\tau\norm{\bf^n}_{L^2(\Omega)}\norm{\bv^n}_{L^2(\Omega)}
\\
&\quad
+\frac{\tau}{\rho}\norm{\bbL}_{\calL(\R^{3\times3}_{\sym})}
\norm{\nabla\bv^n}_{L^2(\Omega)}\norm{\nabla\psi}_{L^2(\Omega)}
+\frac{\tau}{\rho}\norm{\bf^n}_{L^2(\Omega)}\norm{\psi}_{L^2(\Omega)}
\\
&\leq
\calE(\bv^{n-1},\bbE^{n-1})-\calE(\bv^n,\bbE^n)
+c\tau \Lp{1+\norm{\nabla\bv^n}_{L^2(\Omega)}
+\norm{\bf^n}_{L^2(\Omega)}\lp{1+\norm{\bv^n}_{L^2(\Omega)}}}
\end{aligned}
\]
for some constant $c>0$,
where $\bbL$ is the linear mapping defining $\bbS$, compare assumption~\ref{ass:viscstress}.
Here we used 
\[
\wt{\calK}(-\frac{1}{\rho}\psi,0)\calE(\bv^n,\bbE^n)
\leq C\sup_{\rho\norm{\chi}_{\calY^{\bv}}\leq 1}\wt{\calK}(\chi,0)<\infty.
\]
We thus conclude
\begin{equation}
\label{eq:incrementest.2}
\begin{aligned}
\norm{\bv^{n}-\bv^{n-1}}_{(\calY^{\bv})^*}
&\leq\calE(\bv^{n-1},\bbE^{n-1})-\calE(\bv^{n},\bbE^{n})
\\
&\quad
+c\tau \Lp{1+\norm{\nabla\bv^n}_{L^2(\Omega)}
+\norm{\bf^n}_{L^2(\Omega)}\lp{1+\norm{\bv^n}_{L^2(\Omega)}}},
\end{aligned}
\end{equation}
and~\eqref{eq:incrementest.1} leads to
\[
\begin{aligned}
&\norm{\bv_N-\bv_N(\cdot-\tau)}_{L^1(\tau,T;(\calY^{\bv})^*)}
\\
&\ 
\leq \sum_{n=1}^{N}\tau\Lb{\calE(\bv^{n-1},\bbE^{n-1})-\calE(\bv^{n},\bbE^{n})
\\
&\qquad
+c\tau \Lp{1+\norm{\nabla\bv^n}_{L^2(\Omega)}
+\norm{\bf^n}_{L^2(\Omega)}\lp{1+\norm{\bv^n}_{L^2(\Omega)}}}}
\\
&\ 
= \tau\Lp{\calE(\bv^{0},\bbE^{0})-\calE(\bv^{N},\bbE^{N})} + c\tau^2 N
\\
&\qquad
+c\tau\int_{0}^{T}\norm{\nabla\bv_N}_{L^2(\Omega)}\dd t
+c\tau\lp{1+\norm{\bv_N}_{L^\infty(0,T;L^2(\Omega))}}
\int_0^T\norm{\bf_N}_{L^2(\Omega)}\dd t'.
\end{aligned}
\]
Due to $N\tau=T$ and the uniform bounds~\eqref{eq:fN.integral} and~\eqref{eq:aprioriest.v},
this shows~\eqref{eq:aprioriest.increment.v}
and completes the proof.
\end{proof}

Under the additional assumption~\ref{ass:P.gamma},
we can derive an estimate for $\bbE$
that is similar to~\eqref{eq:aprioriest.increment.v}.

\begin{lemma}
\label{lem:increment.est.E}
In the situation of Lemma~\ref{lem:aprioriest}
assume that~\ref{ass:P.gamma} is satisfied,
and that $N\in\N$ is sufficiently large.
Then
\begin{equation}
\label{eq:aprioriest.increment.E}
\norm{\bbE_N-\bbE_N(\cdot-\tau)}_{L^1(\tau,T;(\calY^{\bbE})^*)}
\leq C\tau
\end{equation}
for a constant $C>0$ independent of $N\in\N$.
\end{lemma}

\begin{proof}
We proceed as for the derivation of~\eqref{eq:aprioriest.increment.v},
but we have to take into account the domain of $\calP$.
Let $0<R<R_0$ with $R_0$ as in~\ref{ass:P.gamma}.
We now use~\eqref{eq:weakform.discretized}
and~\eqref{eq:F.est}
for $\psi=0$.
For $\Psi\in\calY^{\bbE}$ with 
$\norm{\Psi}_{\calY^{\bbE}}\leq 1$,
we have $\tau\wt{\calK}(0,-\frac{1}{R}\Psi)\leq 1$ if $N=T/\tau$ is sufficiently large.
Writing $\Psi=R \frac{1}{R}\Psi$, we then have
\[
\begin{aligned}
&\int_\Omega(\bbE^{n}-\bbE^{n-1}):\Psi\dd x
\\
&\quad
\leq 
-R\Lp{1-\tau \wt{\calK}(0,-\frac{1}{R}\Psi)}\calE(\bv^n,\bbE^n)+\calE(\bv^{n-1},\bbE^{n-1})
-\tau R\int_\Omega\gamma|\nabla\bbE|^2-\bf\cdot\bv\dd x
\\
&\qquad\quad-\tau R\Lp{\calP(\dev D\varphi(\bbE^n))
-\calP(-\frac{1}{R}\dev\Psi)}
-\tau\int_\Omega
\bbD(\bv^n):\Psi + \gamma \nabla\bbE^n\tcolon\nabla\Psi 
\dd x
\\
&\quad
\leq
R\Lp{\calE(\bv^{n-1},\bbE^{n-1})-\calE(\bv^n,\bbE^n)}
\\
&\qquad\quad
+c\tau \lp{1+\norm{\nabla\bv^n}_{L^2(\Omega)}+\norm{\nabla\bbE^n}_{L^2(\Omega)}
+\norm{\bf^n}_{L^2(\Omega)}\norm{\bv^n}_{L^2(\Omega)}}
\end{aligned}
\]
for some constant $c>0$,
where we used
\[
\begin{aligned}
\wt{\calK}(0,-\frac{1}{R}\Psi)\calE(\bv^n,\bbE^n)
+\calP(-\frac{1}{R}\dev\Psi)
\leq
\sup_{R\norm{\Phi}_{\calY^{\bbE}}\leq 1}
\Lp{C\wt{\calK}(0,\Phi)
+\calP(\dev\Phi)}
<\infty.
\end{aligned}
\]
We thus obtain
\[
\begin{aligned}
\norm{\bbE^{n}-\bbE^{n-1}}_{({\calY^{\bbE}})^*}
&\leq
R\Lp{\calE(\bv^{n-1},\bbE^{n-1})-\calE(\bv^n,\bbE^n)}
\\
&\qquad
+c\tau \lp{1+\norm{\nabla\bv^n}_{L^2(\Omega)}+\norm{\nabla\bbE^n}_{L^2(\Omega)}
+\norm{\bf^n}_{L^2(\Omega)}\norm{\bv^n}_{L^2(\Omega)}}.
\end{aligned}
\]
This estimate resembles~\eqref{eq:incrementest.2}, 
and we arrive at~\eqref{eq:aprioriest.increment.E} 
along the same lines as in the previous proof.
\end{proof}

\section{Passage to the time-continuous limit}
\label{sec:limit.passage}

From the approximate solutions $(\bv_N,\bbE_N)$ we can extract convergent subsequences.
While we can identify the limit 
to be a weak solution for $\gamma>0$,
the limit object is an energy-variational solution for $\gamma=0$.
In this way, we complete the proofs of Theorem~\ref{thm:existence.gamma}
and Theorem~\ref{thm:existence.envar}.

\subsection{Convergent subsequences}

From the uniform bounds derived in Lemma~\ref{lem:aprioriest} and Lemma~\ref{lem:increment.est.E}, 
we conclude the existence of convergent subsequences.

\begin{lemma}\label{lem:convergence}
Let $\gamma\geq 0$, and
let~\ref{ass:Omega}--\ref{ass:PDphi} be satisfied.
If $\gamma>0$, further assume~\ref{ass:phi.quadratic} and~\ref{ass:P.gamma}.
There exist functions 
\[
\begin{aligned}
\bv&\in L^\infty(0,T;L^2(\Omega;\R^3))\cap L^2(0,T;H^1(\Omega;\R^3)),
\\
\bbE&\in L^\infty(0,T;L^2(\Omega;\R^{3\times3}_{\sym})),
\qquad
E\in\BV([0,T]),
\end{aligned}
\]
and a (not relabeled) subsequence of $(\bv_N,\bbE_N)$
such that
\begin{align}
\bv_N
&\tows\bv 
&&\text{in } L^\infty(0,T;L^2(\Omega;\R^3)),
\label{eq:conv.v.ws}
\\
\nabla\bv_N
&\tow \nabla\bv
&&\text{in } L^2(0,T;L^2(\Omega;\R^{3\times3})),
\label{eq:conv.dv.w}
\\
\bv_N&\to\bv 
&&\text{in }L^p(0,T;L^q(\Omega;\R^3))\text{ for all } p\in[1,\infty),\,q\in[1,6),
\label{eq:conv.v.s}
\\
\bbE_N
&\tows\bbE 
&&\text{in } L^\infty(0,T;L^2(\Omega;\R^{3\times3}_{\sym})),
\label{eq:conv.E.ws}
\\
\calE(\bv_N,\bbE_N)
&\tows E 
&&\text{in }\BV([0,T])
\label{eq:conv.energy}
\end{align}
as $N\to\infty$.
If $\gamma>0$, then it further holds
\[
\bbE\in L^2(0,T;H^1(\Omega;\R^{3\times3}_{\sym})), 
\qquad
E=\calE(\bv,\bbE)\text{ a.e.~in }(0,T),
\]
and
\begin{align}
\nabla\bbE_N
&\tow\nabla\bbE
&&\text{in } L^2(0,T;L^2(\Omega;\R^{3\times3\times3})),
\label{eq:conv.dE.w}
\\
\bbE_N&\to\bbE 
&&\text{in }L^p(0,T;L^q(\Omega;\R^{3\times3}_{\sym})) \text{ for all } p\in[1,\infty),\,q\in[1,6).
\label{eq:conv.E.s}
\end{align}
\end{lemma}

\begin{proof}
The existence of convergent subsequences
as indicated in~\eqref{eq:conv.v.ws}, \eqref{eq:conv.dv.w}, \eqref{eq:conv.E.ws}, \eqref{eq:conv.energy} and, if $\gamma>0$,~\eqref{eq:conv.dE.w}
follows from the 
uniform bounds in Lemma~\ref{lem:aprioriest} and Lemma~\ref{lem:increment.est.E}
by the Banach--Alaoglu theorem.
To conclude the strong convergence~\eqref{eq:conv.v.s},
we combine the convergence~\eqref{eq:conv.dv.w}
with the uniform bound~\eqref{eq:aprioriest.increment.v}
and use the version of the Aubin--Lions--Simon theorem
for piecewise constant functions
due to Dreher and J\"ungel~\cite[Theorem 5]{DreherJuengel2012compactpwconst}.
This yields~\eqref{eq:conv.v.s} for $p=q=2$,
and by interpolation using the uniform bound~\eqref{eq:aprioriest.v}
and the embedding $H^1(\Omega)\hookrightarrow L^6(\Omega)$,
this generalizes to $p\in[1,\infty)$, $q\in[1,6)$.
In the same way, 
we conclude the strong convergence~\eqref{eq:conv.E.s}
from~\eqref{eq:conv.dE.w} and the uniform bounds~\eqref{eq:aprioriest.E} and~\eqref{eq:aprioriest.increment.E}
if $\gamma>0$ and~\ref{ass:P.gamma} hold.
\end{proof}

\begin{remark}
If $\gamma>0$, the convergence~\eqref{eq:conv.dE.w} and~\eqref{eq:conv.E.s}
imply the weak convergence
\[
\nabla[D\varphi(\bbE_N)]
=\Lp{D^2\varphi(\bbE_N)[\partial_j\bbE_N]}_j
\tow\Lp{D^2\varphi(\bbE)[\partial_j\bbE]}_j
=\nabla[D\varphi(\bbE)]
\]
in $L^2(0,T;L^2(\Omega;\R^{3\times3\times3}))$
if $D^2\varphi$ is merely bounded. 
In particular, it is not necessary
to assume that $\varphi$ is quadratic
in order to pass to the limit $N\to\infty$.
Indeed, this assumption is only used for the construction
of approximate solutions by the previous saddle-point argument,
compare Remark~\ref{rem:phi.quadratic.adaption}.
\end{remark}

In order to pass to the limit $N\to\infty$ in~\eqref{eq:discrete.integrated}, 
we further collect elementary convergence properties
of the test functions.
\begin{lemma}
\label{lem:convergence.testfct}
Let $\phi\in C_c^1([0,T))$, $(\psi,\Psi)\in C_c^1([0,T];\calY)$.
Then it holds
\[
\begin{aligned}
\phi_N&\to \phi  
&&\text{in }L^\infty(0,T),
&\partial_t\wh\phi_N&\to \partial_t \phi 
&&\text{in }L^\infty(0,T)
\\
\psi(\cdot+\tau) 1_{(0,T-\tau)}&\to\psi
&&\text{in }L^\infty(0,T;\calY^{\bv}),
&\delta_\tau\psi&\to\partial_t\psi
&&\text{in }L^\infty(0,T;\calY^{\bv}),
\\
\Psi(\cdot+\tau)1_{(0,T-\tau)}&\to\Psi
&&\text{in }L^\infty(0,T;\calY^{\bbE}),
&\delta_\tau\Psi&\to\partial_t\Psi
&&\text{in }L^\infty(0,T;\calY^{\bbE}),
\\
\frac{1}{\tau}\int_0^\tau\psi(t)\dd t &\to \psi(0)
&&\text{in }\calY^{\bv},
&\quad\frac{1}{\tau}\int_0^\tau\Psi(t)\dd t &\to \Psi(0)
&&\text{in }\calY^{\bbE},
\end{aligned}
\]
as $N=\frac{T}{\tau}\to\infty$.
\end{lemma}
\begin{proof}
Due to the uniform continuity of the functions and their derivatives, 
the proof is elementary.
\end{proof}

\subsection{Limit passage for the regularized system}
\label{subsec:limit.weak}

We first focus on the case $\gamma>0$,
where we want to obtain a weak solution 
in the sense of Defintion~\ref{def:weaksol.gamma}.
To this end, we pass to the limit $N\to\infty$ 
in the equation~\eqref{eq:discrete.integrated}.

\begin{proof}[Proof of Theorem~\ref{thm:existence.gamma}]
We use Lemma~\ref{lem:convergence} 
and Lemma~\ref{lem:convergence.testfct}
to pass to the limit in~\eqref{eq:discrete.integrated}.
Since $\gamma>0$, 
we have the weak* convergence $\calE(\bv_N,\bbE_N)\tows\calE(\bv,\bbE)$ in $\BV([0,T])$.
Together with the weak* convergence $\bv_N\tows\bv$ and $\bbE_N\tows\bbE$
from~\eqref{eq:conv.v.ws} and~\eqref{eq:conv.E.ws},
the fact that $\bf_N\to\bf$ strongly in $L^1(0,T;L^2(\Omega))$,
and the convergence properties from Lemma~\ref{lem:convergence.testfct},
the limit passage in the first, second and fourth line of~\eqref{eq:discrete.integrated}
is direct.
For the third line 
we invoke the weak lower semicontinuity of norms and of $\calP$ by assumption~\ref{ass:P},
so that we have
\[
\begin{aligned}
\liminf_{N\to\infty}
&\int_0^T\phi_N\Lp{\calD_\gamma(\bv_N,\bbE_N)
+\calP(\dev D\varphi(\bbE_N))
-\calP(\dev\Psi)}\dd t
\\
&\geq \int_0^T\phi\Lp{\calD_\gamma(\bv,\bbE)
+\calP(\dev D\varphi(\bbE))
-\calP(\dev\Psi)}\dd t.
\end{aligned}
\]
It remains to discuss the term related with $\calN_\gamma$ defined in~\eqref{eq:DN.def}.
By the weak convergence properties established in Lemma~\ref{lem:convergence},
we can directly pass to the limit in the linear terms.
Analogously, 
the strong convergence from~\eqref{eq:conv.E.s}
allows to pass to the limit
in the terms related with $\varphi$ and $D\varphi$,
which are now quadratic and linear functions, respectively. 
For the remaining terms,
which are quadratic in $(\bv,\bbE)$,
we can use a standard argument 
combining the weak convergence from~\eqref{eq:conv.dv.w}
with the strong convergence from~\eqref{eq:conv.v.s} and~\eqref{eq:conv.E.s}.

In summary, passage to the limit in~\eqref{eq:discrete.integrated}
leads to
\begin{equation}
\label{eq:weakform.variational.pos}
\begin{aligned}
&\int_0^T-\partial_t\phi\Lb{
\calE(\bv,\bbE)
-\int_\Omega\rho\bv\cdot\psi + \bbE:\Psi \dd x
}\dd t
\\
&\quad
+\int_0^T\phi(t)\int_\Omega\rho\bv\cdot\partial_t\psi
+\bbE:\partial_t\Psi\dd x\dd t
\\
&\quad
+\int_0^T\phi\Lp{\calD_\gamma(\bv,\bbE)
+\calP(\dev D\varphi(\bbE))
+\calN_\gamma(\bv,\bbE\mid\psi,\Psi)
-\calP(\dev\Psi)}\dd t
\\
&\quad-\phi(0)\Lb{\calE(\bv_0,\bbE_0)-\int_\Omega\bv_0\cdot\psi(0)+\bbE_0:\Psi(0)\dd x}
\leq 0
\end{aligned}
\end{equation}
for all $(\psi,\Psi)\in C^1([0,T];\calY)$
and $\phi\in C_c^1([0,T))$ with $\phi\geq 0$.
In virtue of Lemma~\ref{lem:ineq.pointwisevariational},
this means that $(\bv,\bbE)$ is a weak solution in the sense of Defintion~\ref{def:weaksol.gamma},
compare Remark~\ref{rem:weakform.variational}.
\end{proof}

\subsection{Limit passage for the original system}
\label{subsec:limit.envar}

Now we consider the case $\gamma=0$,
where we want to obtain an energy-variational solution in the sense of 
Definition~\ref{def:envar}.
As 
we have neither strong convergence of $\bbE$
nor the identity $E=\calE(\bv,\bbE)$
for $\gamma=0$, 
we make use of the regularity weight $\calK$ from~\eqref{eq:K.lsc}
to pass to the limit in~\eqref{eq:discrete.integrated} by weak lower semicontinuity.

We define $E_N\coloneqq \calE(\bv_N,\bbE_N)$ in $(0,T)$ 
and $E_0\coloneqq \calE(\bv_0,\bbE_0)$,
and we write~\eqref{eq:discrete.integrated} as
\begin{equation}
\label{eq:discrete.integrated.convexified}
\begin{aligned}
&\int_\tau^T-\partial_t\wh\phi_N(t)\Lb{
E_N(t)
-\int_\Omega\rho\bv_N(t)\cdot\psi(t+\tau) + \bbE_N(t):\Psi(t+\tau)\dd x
}\dd t
\\
&\ 
+\int_\tau^T\phi_N\Lp{\int_\Omega\rho\bv_N\cdot\delta_\tau\psi
+\bbE_N:\delta_\tau\Psi\dd x
+\calP(\dev D\varphi(\bbE_N))-\calP(\dev\Psi))}\dd t
\\
&\ 
-\int_0^T\phi_N\calB_N(\bv_N-\psi)\dd t
\\
&\ 
+\int_0^T\phi_N\Lp{\calD_0(\bv_N)
+\calN_0(\bv_N,\bbE_N\mid\psi,\Psi)
+\mathcal K(\psi,\Psi)\lp{\calE(\bv_N,\bbE_N)-E_N}
}\dd t
\\
&\
+\phi(0)\Lb{E_0
-\int_\Omega\bv_0\cdot\Lp{\frac{1}{\tau}\int_0^\tau\psi(t)\dd t}
+\bbE_0:\Lp{\frac{1}{\tau}\int_0^\tau\Psi(t)\dd t} \dd x}
\leq 0.
\end{aligned}
\end{equation}
Observe that in the fourth line,
we introduced a term that equals zero and is related to $\calK$.
Taking $\calK$ as in~\eqref{eq:K.lsc}
ensures convexity of certain terms,
so that we can pass to the limit $N\to\infty$ in inequality~\eqref{eq:discrete.integrated.convexified} 
by lower semicontinuity. 
More precisely, we pass to the limit in the second part of this new term
by using the weak* convergence of $(E_N)$ in $\BV([0,T])$,
while the first term combines with the certain terms
to ensure weak lower semicontinuity 
by the following convexity property.

\begin{lemma}\label{lem:K.giveslsc}
Let $\calK$ be as in~\eqref{eq:K.lsc}. For each $(\psi,\Psi)\in\calY$, the functional
$\calJ_{(\psi,\Psi)}\colon\calX_0\to\R$ with
\[
\begin{aligned}
\calJ_{(\psi,\Psi)}(\bv,\bbE)&=
\int_\Omega-\np{D\varphi(\bbE)+\varphi(\bbE)\bbI}:\nabla\psi 
+ (\dv\bv)\bbE:\Psi -(\bbE\bbW(\bv)-\bbW(\bv)\bbE):\Psi
\dd x
\\
&\quad+\calD_0(\bv)
+\mathcal K(\psi,\Psi)\calE(\bv,\bbE)
\end{aligned}
\]
is convex and lower semicontinuous, and thus weakly lower semicontinous.
\end{lemma}

\begin{proof}
We split the functional $\calJ_{(\psi,\Psi)}$
into several parts.
Firstly, from~\ref{ass:phi} we obtain the 
convexity of
\[
\begin{aligned}
\bbE&\mapsto-\varphi(\bbE)\dv\psi + \norm{\dv\psi}_{L^\infty(\Omega)}\varphi(\bbE)
=\lp{\norm{\dv\psi}_{L^\infty(\Omega)}-\dv\psi}\varphi(\bbE),
\\
\bbE&\mapsto -D\varphi(\bbE):\nabla\psi 
+ \frac{M}{\kappa}\norm{\nabla\psi}_{L^\infty(\Omega)}\varphi(\bbE)
\\
& \quad
=-D\varphi(\bbE):\nabla\psi 
+ \frac{M}{\kappa}\norm{\nabla\psi}_{L^\infty(\Omega)}\frac{\kappa}{2}|\bbE|^2
+\frac{M}{\kappa}\norm{\nabla\psi}_{L^\infty(\Omega)}\Lp{\varphi(\bbE)-\frac{\kappa}{2}|\bbE|^2},
\end{aligned}
\]
which implies convexity and lower semicontinuity of the associated integral functional.
Secondly, by Young's inequality, we have
\[
\begin{aligned}
(\bv,\bbE)&\mapsto
(\dv\bv)\bbE:\Psi
+\frac{1}{2}\bbS(\bv):\bbD(\bv) 
+\frac{1}{\kappa\mu}\norm{\Psi}_{\infty}^2\,\varphi(\bbE)
\\
&
\geq -\abs{\nabla\bv}\,\abs{\bbE}\,\norm{\Psi}_\infty 
+\frac{\mu}{2}\abs{\nabla\bv}^2 
+ \frac{1}{2\mu}\norm{\Psi}_{\infty}^2\abs{\bbE}^2
+ \frac{1}{\kappa\mu}\norm{\Psi}_{L^\infty(\Omega)}^2\Lp{\varphi(\bbE)-\frac{\kappa}{2}\abs{\bbE}^2}
\geq 0, 
\\
(\bv,\bbE)&\mapsto
-(\bbE\bbW(\bv)-\bbW(\bv)\bbE):\Psi
+\frac{1}{2}\bbS(\bv):\bbD(\bv)
+\frac{4}{\kappa\mu}\norm{\Psi}_{\infty}^2\,\varphi(\bbE)
\\
&
\geq-2\abs{\bbE}\,\abs{\nabla\bv}\,\norm{\Psi}_\infty
+\frac{\mu}{2}\abs{\nabla\bv}^2 
+\frac{2}{\mu}\norm{\Psi}_{\infty}^2\abs{\bbE}^2
+\frac{4}{\kappa\mu}\norm{\Psi}_{\infty}^2\Lp{\varphi(\bbE)-\frac{\kappa}{2}\abs{\bbE}^2}
\geq 0.
\end{aligned}
\]
Therefore, both functions are quadratic and nonnegative, 
and thus convex.
This implies the weak lower semicontinuity of the associated integral functional. 
Summing up and recalling the definition of $\calD_\gamma$ 
from~\eqref{eq:DN.def}
with $\gamma=0$,
wee conclude that $\calJ_{(\psi,\Psi)}$
is also convex and lower semicontinuous.
\end{proof}

\begin{remark}
As $\varphi$ is non-negative, the function
\[
\bbE\mapsto-\varphi(\bbE)\dv\psi + \norm{\max\set{\dv\psi,0}}_{L^\infty(\Omega)}\varphi(\bbE)
\]
is also convex. Therefore, one could improve $\mathcal K$ from~\eqref{eq:K.lsc}
by taking into account only the non-negative part of $\dv\psi$.
Such an optimized regularity weight 
would be in line with the choices considered in~\cite{EitLas24envarhyp,EiterSchindler25}
for instance.
However, a similar improvement seems to be unavailable for the other terms in~\eqref{eq:K.lsc},
which is why we omit it here.
\end{remark}

We can now finalize the proof of Theorem~\ref{thm:existence.envar}.

\begin{proof}[Proof of Theorem~\ref{thm:existence.envar}]
We use the convergence properties established in Lemma~\ref{lem:convergence}
(for $\gamma=0$) 
and Lemma~\ref{lem:convergence.testfct}
in order to pass to the limit in~\eqref{eq:discrete.integrated.convexified}.
The first, third and fifth line in~\eqref{eq:discrete.integrated.convexified}
are linear in $\bv_N$ and $\bbE_N$,
and we can directly pass to the limit.
For the second line of~\eqref{eq:discrete.integrated.convexified},
we combine a similar argument with assumption~\ref{ass:PDphi},
which yields
\[
\liminf_{N\to\infty}\int_\tau^T\phi_N\calP(\dev D\varphi(\bbE_N)) \dd t
\geq \int_0^T\phi\,\calP(\dev D\varphi(\bbE)) \dd t.
\]
To pass to the limit in the fourth line of inequality~\eqref{eq:discrete.integrated.convexified},
we first use the strong convergence from~\eqref{eq:conv.v.s}.
A combination with Lemma~\ref{lem:convergence.testfct}
and the weak and weak* convergence from~\eqref{eq:conv.dv.w} and~\eqref{eq:conv.E.ws}
yields
\[
\begin{aligned}
&\lim_{N\to\infty}\int_0^T\phi_N\int_\Omega
\frac{\rho}{2}(\dv\bv_N)(\bv_N\cdot\psi)
+\rho (\bv_N\otimes\bv_N){:}\nabla\psi 
\dd x \dd t
\\
&\qquad+\lim_{N\to\infty}\int_0^T\phi_N\int_\Omega
-\bbS(\bv_N){:}\bbD(\psi)
+\bv_N\cdot\nabla \Psi{:}\bbE_N
+\bbD(\bv_N){:}\Psi\dd x \dd t
\\
&=\int_0^T\phi\int_\Omega
\frac{\rho}{2}(\dv\bv)(\bv\cdot\psi)
+\rho (\bv\otimes\bv){:}\nabla\psi 
-\bbS(\bv){:}\bbD(\psi)
+\bv\cdot\nabla \Psi{:}\bbE
+\bbD(\bv){:}\Psi\dd x \dd t
\end{aligned}
\]
Combining the convexity and lower semicontinuity of $\calJ_{(\psi,\Psi)}$
by Lemma~\ref{lem:K.giveslsc}
with Fatou's lemma,
we also deduce weak lower semicontinuity of the time-integrated functional,
which implies
\[
\liminf_{N\to\infty}\int_0^T \phi_N\calJ_{(\psi,\Psi)}(\bv_N,\bbE_N)\dd t
\geq\int_0^T \phi\,\calJ_{(\psi,\Psi)}(\bv,\bbE)\dd t.
\]
Moreover, 
from~\eqref{eq:conv.energy} we conclude
\[
\lim_{N\to\infty}\int_0^T -\phi_N \calK(\psi,\Psi)E_N
=\int_0^T -\phi \,\calK(\psi,\Psi)E.
\]
Recalling the definition of
and $\calN_\gamma$ from~\eqref{eq:DN.def} (with $\gamma=0$),
we can thus pass to the limit inferior in the fourth line of~\eqref{eq:discrete.integrated.convexified}.
We thus obtain
\begin{equation}
\label{eq:weakform.variational.zero}
\begin{aligned}
&\int_0^T-\partial_t\phi\Lb{
E
-\int_\Omega\rho\bv\cdot\psi + \bbE:\Psi \dd x
}\dd t
\\
&\quad
+\int_0^T\phi\Lp{\int_\Omega\rho\bv\cdot\partial_t\psi
+\bbE:\partial_t\Psi\dd x
+\calP(\dev D\varphi(\bbE))
-\calP(\dev\Psi)}\dd t
\\
&\quad
+\int_0^T\phi\Lp{
-\calB(\bv-\psi)
+\calD_0(\bv)
+\calN_0(\bv,\bbE\mid\psi,\Psi)
+\mathcal K(\psi,\Psi)\lp{\calE(\bv,\bbE)-E}
}\dd t
\\
&\quad+\phi(0)\Lb{E_0
-\int_\Omega\bv_0\cdot\psi(0)+\bbE_0:\Psi(0)\dd x}
\leq 0.
\end{aligned}
\end{equation}
Employing Lemma~\ref{lem:ineq.pointwisevariational},
we deduce that $(\bv,\bbE,E)$ is an energy-variational solution 
in the sense of Definition~\ref{def:envar}
with $E(0)=\calE(\bv_0,\bbE_0)$.
\end{proof}

\subsection*{Acknowledgments}
This research has been funded by Deutsche
Forschungsgemeinschaft (DFG) through grant CRC 1114 ``Scaling Cascades in
Complex Systems'', Project Number 235221301, Project YIP,
and supported through grant SPP 2410 ``Hyperbolic Balance Laws in Fluid Mechanics:~Complexity, Scales, Randomness (CoScaRa)'', Project Number
526018747.
Moreover, the author thanks Tom\'{a}\v{s}  Roub\'{i}\v{c}ek for several helpful discussions during the initiation phase of this project.

\bibliographystyle{plainurl}
\bibliography{literature}

\end{document}